\documentclass{amsart}
\usepackage{amscd}
\numberwithin{equation}{section}
\usepackage{amsmath}
\usepackage{latexsym}
\usepackage{graphicx}
\usepackage{subfigure}
\usepackage{color}
\usepackage{psfrag}
\usepackage{amssymb,amsfonts,amstext,graphicx}
\usepackage[curve]{xypic}
\usepackage[all]{xy}

\let\cal\mathcal

\def\Sym{\mathop{\text{\upshape{Sym}}}\nolimits}

\def\diag{\operatorname {diag}}

\newtheorem{lemma}{Lemma}[section]
\newtheorem{proposition}[lemma]{Proposition}
\newtheorem{theorem}[lemma]{Theorem}

\newtheorem{corollary}[lemma]{Corollary}

\newtheorem{fact}[lemma]{Fact}

\theoremstyle{definition}

\newtheorem{example}[lemma]{Example}
\newtheorem{definition}[lemma]{Definition}

\newtheorem{notation}[lemma]{Notation}

{

\newtheorem*{prooftheorem1}{Proof of Theorem \ref{theorem1}}
\newtheorem*{prooftheorem3}{Proof of Theorem \ref{theorem3}}
\newtheorem*{sketchofproof}{Sketch of the proof}

}

\theoremstyle{remark}

\newtheorem{remark}[lemma]{Remark}

\newcommand{\Acal}{\mbox{$\cal A$}}

\newcommand{\Dcal}{\mbox{$\cal D$}}

\newcommand{\Lcal}{\mbox{$\cal L$}}

\newcommand{\Ncal}{\mbox{$\cal N$}}

\newcommand{\Ocal}{\mbox{$\cal O$}}

\newcommand{\Rcal}{\mbox{$\cal R$}}

\newdimen\uboxsep \uboxsep=1ex
\def\uboxn#1{\vtop to 0pt{\hrule height 0pt depth 0pt\vskip\uboxsep
\hbox to 0pt{\hss #1\hss}\vss}}

\def\uboxs#1{\vbox to 0pt{\vss\hbox to 0pt{\hss #1\hss}
\vskip\uboxsep\hrule height 0pt depth 0pt}}

\title[PBW algebras, YBE and AS regularity]{Quadratic algebras, Yang-Baxter equation,
and Artin-Schelter regularity} 
\thanks{The author was partially supported by the
 The Abdus Salam International Centre
   for Theoretical Physics (ICTP), Trieste}
\author{Tatiana Gateva-Ivanova}
\address{Institute of Mathematics and Informatics\\
Bulgarian Academy of Sciences\\
Sofia 1113, Bulgaria\\
and
American University in Bulgaria\\
2700 Blagoevgrad, Bulgaria }

\email{tatianagateva@yahoo.com, tatyana@aubg.bg}

\begin{document}
\date{\today}
\begin{abstract}
We study quadratic algebras over a field $\textbf{k}$. We show
that an $n$-generated PBW algebra $A$ has finite global dimension
and polynomial growth \emph{iff} its Hilbert series is $H_A(z)= 1
/(1-z)^n$. Surprising amount can be said when the algebra $A$ has
\emph{quantum binomial relations}, that is the defining relations
are nondegenerate square-free binomials $xy-c_{xy}zt$ with
non-zero coefficients $c_{xy}\in \textbf{k}$.  In this case
various good algebraic and homological properties are closely
related. The main result shows that for an $n$-generated quantum
binomial algebra $A$ the following conditions are equivalent: (i)
A is a PBW algebra with finite global dimension; (ii) A is PBW and
has polynomial growth; (iii) A is an Artin-Schelter regular PBW
algebra; (iv) $A$ is a Yang-Baxter algebra; (v) $H_A(z)=
1/(1-z)^n;$ (vi) The dual $A^{!}$ is a quantum Grassman algebra;
(vii) A is a binomial skew polynomial ring. So for quantum
binomial algebras the problem of classification of Artin-Schelter
regular PBW algebras of global dimension $n$ is equivalent to the
classification of square-free set-theoretic solutions of the
Yang-Baxter equation $(X,r)$, on sets $X$ of order $n$.
\end{abstract}
\maketitle
\section{Introduction}
We work with quadratic algebras $A$ over a ground field
$\textbf{k}$. Following a classical tradition (and recent trend),
we take a combinatorial approach to study $A$. The properties of
$A$ will be read off a presentation $A= \textbf{k} \langle
X\rangle /(\Re)$, where $X$ is a finite set of generators of
degree $1$, $|X|=n,$ $\textbf{k}\langle X\rangle$ is the unitary
free associative algebra generated by $X$, and $(\Re)$ is the
two-sided ideal of relations, generated by a {\em finite} set
$\Re$ of homogeneous polynomials of degree two. Clearly $A$ is a
connected graded $\textbf{k}$-algebra (naturally graded by length)
 $A=\bigoplus_{i\ge0}A_i$, where
$A_0=\textbf{k}$, $A$ is generated
by $A_1=Span_{\textbf{k}}X,$
so each $A_i$ is finite dimensional.

A quadratic algebra $A$ is  \emph{a PBW algebra} if  there exists
an enumeration of $X,$ $X= \{x_1, \cdots x_n\}$ such that the
quadratic relations $\Re$ form a (noncommutative) Gr\"{o}bner
basis with respect to the degree-lexicographic ordering on
$\langle X\rangle$ induced from $x_1 < x_2< \cdots <x_n$. In this
case   the set of normal monomials (mod $\Re$) forms a
$\textbf{k}$-basis of $A$ called a \emph{PBW basis}
 and $x_1,\cdots, x_n$ (taken exactly with this enumeration) are called \emph{ PBW-generators of $A$}.
 The notion of a PBW algebra was introduced by Priddy, \cite{priddy}, his
 \emph{PBW basis}  is a generalization of the classical Poincar\'{e}-Birkhoff-Witt basis for the universal enveloping of a  finite dimensional Lie algebra.
 PBW algebras form an important class of Koszul algebras.
  The interested reader can find information on PBW algebras and more references in \cite{PP}.
One of the central problems that we consider is
\[
\text{\emph{the classification of Artin-Schelter regular PBW
algebras.}}
\]
It is far from its final resolution.
The first question to be asked is
\[\begin{array}{l}
\text{\emph{What can be said about PBW algebras with polynomial growth and finite global}}\\
\text{\emph{dimension? }}
\end{array}
\]
We find that, surprisingly, the  class $\mathfrak{C}_n$ of
$n$-generated PBW algebras with polynomial growth and finite
global dimension is determined uniquely by its Hilbert series,
this is in section \ref{section_growthgldim}.

 \begin{theorem}
\label{theorem1} Let $A  = \textbf{k}\langle X  \rangle /(\Re)$ be
a quadratic PBW algebra, where $X=\{x_1, x_2, \cdots, x_n\}$ is a
set of PBW generators.
 The following are equivalent
\begin{enumerate}
\item
\label{t1gldimpolgr}
$A$ has polynomial growth and finite  global dimension.
 \item
\label{t1gldim}
$A$ has  exactly $\binom{n}{2}$ relations  and finite  global dimension.
 \item
 \label{t1hilbser}
The Hilbert series of $A$ is
\[
H_A(z)= \frac{1}{(1-z)^n}.
\]
 \item
\label{t1perm}
There exists a permutation
$y_1, \cdots, y_n$ of $x_1 \cdots x_n$, such that the set
\begin{equation}
\label{ncal} \Ncal = \{y_1^{\alpha_{1}}y_2^{\alpha_{2}}\cdots
y_n^{\alpha_{n}}\mid \alpha_{i} \geq 0 \text{ for } 1 \leq i
\leq n\}
\end{equation}
is a $\textbf{k}$-basis of $A.$
\end{enumerate}
Furthermore the class $\mathfrak{C}_n$ of all $n$-generated PBW
algebras with polynomial growth and finite global dimension
contains a unique (up to isomorphism) monomial algebra:
 \[
 A^{0}= \langle x_1, \cdots, x_n \rangle/(x_jx_i \mid 1 \leq i < j \leq n).
 \]
\end{theorem}
Note that $y_1, y_2, \cdots, y_n$ is possibly a "new" enumeration
of $X$, which  induces a degree-lexicographic ordering $\prec$ on
$\langle X \rangle$ (with $y_1 \prec y_2\prec \cdots \prec y_n$)
different from the original ordering. The defining relations
remain the same, but their leading terms w.r.t. $\prec$ may be
different from the original ones, and $y_1, y_2, \cdots, y_n$ are
not necessarily PBW generators of $A$. In the terminology  of
Gr\"{o}bner bases, $\Ncal$ is not necessarily \emph{a normal
basis} of $A$ w.r.t. $\prec$.

A class of PBW Arin Schelter regular rings of arbitrarily high
global dimension $n$,  were introduced and studied in \cite{T96},
\cite{TM}, \cite{T04}, \cite{T04s}. These are \emph{the binomial
skew-polynomial rings}. It was shown in \cite{TM} that they are
also closely related to the set-theoretic solutions of the
Yang-Baxter equation. So  we consider the so-called \emph{quantum
binomial algebras} introduced and studied in \cite{T04s},
\cite{TSh08}. These are  quadratic algebras (not necessarily PBW)
with square-free non-degenerate binomial relations, see Definition
\ref{quantumbinomialalg_def}. The second question that we ask in
the paper is
\[\begin{array}{l}
\text{\emph{Which are the PBW Artin Schelter regular algebras in the class of}}\\
\text{\emph{quantum binomial algebras?}}\\
\end{array}
\]
We prove that each quantum binomial PBW algebra with finite global
dimension is a Yang-Baxter algebra, and therefore a binomial
skew-polynomial ring. This implies that \emph{in the class of
quantum binomial algebras} the three notions: an Artin-Schelter
regular PBW algebra, a binomial skew-polynomial ring, and a
Yang-Baxter algebra (in the sense of Manin) are equivalent.
  The following result is proven in Section \ref{section_AS_YBE}.
\begin{theorem}
\label{theorem3} Let $A  = \textbf{k}\langle X  \rangle /(\Re)$ be
a quantum binomial algebra.
The following conditions are equivalent.
 \begin{enumerate}
\item \label{theorem31} $A$ is a PBW algebra with  finite global
dimension. \item \label{theorem32} $A$ is a PBW algebra with
polynomial growth. \item \label{theorem33}
 $A$ is  an Artin-Schelter regular PBW algebra.
\item
 \label{theorem34}
 $A$ is a Yang-Baxter algebra, that is the set of relations $\Re$ defines
 canonically  a solution of the
Yang-Baxter equation.
\item
\label{theorem35} $A$ is a binomial skew polynomial ring, with
respect to some appropriate enumeration of $X.$
 \item
\label{theorem36}
\[\dim_{\textbf{k}} A_3 = \binom{n+2}{3},\quad\text{or equivalently}\quad\dim_{\textbf{k}} A^{!}_3 = \binom{n}{3}.\]
 \item
\label{theorem37}
\[H_A(z)= \frac{1}{(1-z)^n}\]
 \item
\label{theorem38}
The Koszul dual $A^{!}$ is a quantum Grassman algebra.
\end{enumerate}
Each of these condition implies that $A$ is Koszul and a Noetherian domain.
\end{theorem}

It follows from Theorem \ref{theorem3} that
\[
\begin{array}{l}
\text{\emph{The problem of classification of Artin Schelter regular PBW algebras }} \\
\text{\emph{with quantum binomial relations
 and global dimension $n$ is equivalent to}}\\
 \text{\emph{the classification of square-free set-theoretic solutions of YBE, $\;(X,r),$
 }}\\
 \text{\emph{on sets $X$ of order  $n$.}}
\end{array}
\]
Even  under these strong restrictions on the shape of the relations, the problem is  highly nontrivial.
However for reasonably small $n$ (say $n\leq 10$) the square-free solutions of YBE $(X,r)$ are known.
A possible classification for general $n$ can be based on the so called \emph{multipermutation level} of  the solutions, see \cite{TP}.

The paper is organized as follows.

In section \ref{section_preliminaries} are recalled basic
definitions, and some facts used throughout the paper. In section
\ref{section_growthgldim} we study the general case of PBW
algebras with finite global dimension and polynomial growth and
prove Theorem \ref{theorem1}. The approach is combinatorial. To
each PBW algebra $A$ we associate two finite oriented graphs. The
first is \emph{the graph  of normal words } $\Gamma_{\textbf{N}}$
(this is a particular case of the Ufnarovski graph
\cite{Ufnarovski}), it determines the growth and the Hilbert
series of $A$. The second is \emph{the graph of obstructions,}
$\Gamma_{\textbf{W}}$, dual to $\Gamma_{\textbf{N}}$.  We define
it via the set of obstructions (in the sense of Anick,
\cite{Anick85}, \cite{Anick86}),  it gives a precise information
about \emph{the global dimension} of the algebra $A$. We prove
that all  algebras in the class $\mathfrak{C}_n$ of $n$- generated
PBW algebras determine a unique (up to isomorphism) graph of
obstructions $\Gamma_W$, which is the complete oriented graph
$K_n$ with no cycles. The two graphs play important role in the
proof of Theorem \ref{theorem1}. They  can be used whenever PBW
algebras are studied.

In section \ref{section_combinatorics} we find some interesting combinatorial results
on quantum binomial  sets $(X,r)$ and the corresponding quadratic algebra $\Acal=\Acal(\textbf{k}, X,r)$. We study the action of  the infinite Dihedral group, $\Dcal= \Dcal(r)$,
associated with $r$ on $X^3$ and find some \emph{counting formulae} for the $\Dcal$-orbits.
These are used to show that $(X,r)$ is a set-theoretic solution of the Yang-Baxter equation
\emph{iff} $\dim_{\textbf{k}}\Acal_3=\binom{n+2}{3}$.

In section \ref{section_AS_YBE}
we prove Theorem \ref{theorem3}.
  The proof involves the results of sections \ref{section_growthgldim}, \ref{section_combinatorics},
and  results  on binomial skew polynomial rings and set-theoretic solutions of YBE from
\cite{TM}, \cite{T04}, \cite{T04s}, \cite{TSh08}.

\section{Preliminaries - some  definitions and facts}
\label{section_preliminaries}

In this section we recall basic
notions and results which will be used in the paper. This paper is a natural
continuation of \cite{T04s}.
 We shall use the terminology,  notation and  results from
 our previous works
\cite{T96, TM, T04s,  T04, TSh08}.
 The reader  acquainted with these can proceed  to the next section.

A connected graded algebra is called \emph{Artin-Schelter regular} (or \emph{AS regular}) if
\begin{enumerate}
\item[(i)]
$A$ has {\em finite global dimension\/} $d$, that is, each graded
$A$-module has a free resolution of length at most $d$.
\item[(ii)]
$A$ has \emph {finite Gelfand-Kirillov dimension}, meaning that
the integer-valued function $i\mapsto\dim_{\textbf{k}}A_i$ is
bounded by a polynomial in $i$.
\item[(iii)]
$A$ is \emph{Gorenstein}, that is, $Ext^i_A(\textbf{k},A)=0$ for
$i\ne d$ and $Ext^d_A(\textbf{k},A)\cong \textbf{k}$.
\end{enumerate}

AS regular algebras were introduced  and studied first in
\cite{AS}, \cite{ATV1}, \cite{ATV2}. When $d\le3$ all regular
algebras are classified. Since then AS regular algebras and their
geometry are intensively studied. The problem of classification of
regular rings is difficult and remains open even for regular rings
of global dimension $4$. The study of Artin-Schelter regular
rings, their classification, and finding new classes of such rings
is one of the basic problems for noncommutative geometry. Numerous
works on this topic appeared during the last two decades, see for
example \cite{ATV1, ATV2}, \cite{BSM},
\cite{Levasseur},
\cite{Michel-Tate}, \cite{LPWZ}, et all.

A class of PBW AS regular algebras of global dimension $n$ was
introduced and studied in \cite{T96}, \cite{TM} \cite{T04},
\cite{T04s}. These are \emph{the binomial skew-polynomial rings}.
\begin{definition}
\label{binomialringdef}
\cite{T96} A {\em binomial skew polynomial ring\/} is a quadratic algebra
 $A=\textbf{k} \langle x_1, \cdots , x_n\rangle/(\Re)$ with
precisely $\binom{n}{2}$ defining relations \[\Re=\{x_{j}x_{i} -
c_{ij}x_{i^\prime}x_{j^\prime}\}_{1\leq i<j\leq n}\] such that
\begin{enumerate}
\item[(a)] \label{binomialringa} $c_{ij} \in \textbf{k}^{\times}$.

\item[(b)] \label{binomialringb} For every pair $i, j, \; 1\leq
i<j\leq n$, the relation $x_{j}x_{i} - c_{ij}x_{i'}x_{j'}\in \Re,$
satisfies $j
> i^{\prime}$, $i^{\prime} < j^{\prime}$;
\item[(c)] \label{binomialringc} Every ordered monomial $x_ix_j,$
with $1 \leq i < j \leq n$ occurs in  the right hand side of some
relation in $\Re$; \item[(d)] \label{binomialringd} $\Re$ is the
{\it reduced Gr\"obner basis\/} of the two-sided ideal $(\Re)$,
with respect to the order $\prec$ on $\langle X \rangle$,  or
equivalently the ambiguities $x_kx_jx_i,$ with $k>j>i$ do not give
rise to new relations in $A.$
\end{enumerate}

We say that $\Re$ are \emph{relations of skew-polynomial type} if
conditions \ref{binomialringdef} (a), (b) and (c) are satisfied
(we do not assume (d)).
\end{definition}

By \cite{Bergman} condition \ref{binomialringdef} (d) may be
rephrased by saying that  \emph{the set of ordered monomials}
\[
 \Ncal _0 = \{x_{1}^{\alpha_{1}}\cdots
x_{n}^{\alpha_{n}}\mid \alpha_{n} \geq 0 \text{ for } 1 \leq i
\leq n\}
\]
is a $\textbf{k}$-basis of $A$.

\begin{remark}
In the terminology of this paper \emph{a binomial skew polynomial
ring} is a quadratic PBW algebra $A$ with PBW generators $x_1
\cdots, x_n$ and relations of skew-polynomial type.
\end{remark}
More generally, we will consider a class of quadratic algebras
with binomial relations, called \emph{quantum binomial algebras},
these are not necessarily PBW algebras.

We need to recall first the notions of a quadratic set  and the associated with it quadratic algebras.
\begin{definition}
\label{quadraticsetdef}
Let $X$ be a
nonempty set and let $r: X\times X \longrightarrow X\times X$ be a
bijective map. In this case  we shall use notation $(X,r)$ and
refer to it as a \emph{quadratic set}.
We present the
image of $(x,y)$ under $r$ as
\begin{equation}
\label{r} r(x,y)=({}^xy,x^{y}).
\end{equation}
The formula (\ref{r}) defines a ``left action" $\Lcal: X\times X
\longrightarrow X,$ and a ``right action" $\Rcal: X\times X
\longrightarrow X,$ on $X$ as:
\begin{equation}
\label{LcalRcal} \Lcal_x(y)={}^xy, \quad \Rcal_y(x)= x^{y},
\end{equation}
for all $ x, y \in X.$
$r$ is {\em nondegenerate} if  the maps $\Lcal_x$ and  $\Rcal_x$ are
bijective for each $x\in X$.
$r$ is \emph{involutive} if $r^2 = id_{X \times X}$.

As a notational tool, we  shall often
identify the sets $X^{\times k}$ of ordered $k$-tuples, $k \geq 2,$  and $X^k,$ the set of
all monomials of length $k$ in the free monoid $\langle
X\rangle.$
\end{definition}

As in \cite{T04, TSh07, TSh08, TSh0806, TP}, to each quadratic set
$(X,r)$ we associate canonically algebraic objects (see
Definition~\ref{associatedobjects}) generated by $X$ and with
quadratic defining relations $\Re_0 =\Re_0(r)$ naturally determined as
\begin{equation}
\label{defrelations}
\begin{array}{ll}
xy=y^{\prime} x^{\prime}\in \Re_0(r),&\text{whenever} \\
 r(x,y) = (y^{\prime}, x^{\prime}) & \text{and} \quad (x,y) \neq (y^{\prime}, x^{\prime})\;\;\text{hold in}\;\;X \times X.
\end{array}
\end{equation}
We can tell the precise number of defining relations,  whenever
$(X,r)$  is nondegenerate and involutive, with $\mid X \mid =n$.
In this case the nondegeneracy implies that
$\Re(r)$,  consists of precisely $\binom{n}{2}$ quadratic
 relations (see \cite{T10} Proposition 2.3 ).

\begin{definition} \cite{T04, TSh08}
\label{associatedobjects} Assume that $r:X^2 \longrightarrow X^2$
is a bijective map.

(i) The monoid
\[
S =S(X, r) = \langle X ; \; \Re_0(r) \rangle,
\]
 with a
set of generators $X$ and a set of defining relations $ \Re(r),$
is called \emph{the monoid associated with $(X, r)$}.
 The \emph{group $G=G(X, r)$ associated with} $(X, r)$ is
defined analogously.

(ii) For arbitrary fixed field $\textbf{k}$, \emph{the
$\textbf{k}$-algebra associated with} $(X ,r)$ is defined as
\begin{equation}
\label{Adef} \Acal = \Acal(\textbf{k},X,r) = \textbf{k}\langle X ; \;\Re_0(r) \rangle \simeq \textbf{k}\langle X  \rangle /(\Re),
\end{equation}
where $\Re = \Re(r)$ is the set of quadratic \emph{binomial relations}
\begin{equation}
\label{Re0} \Re = \{xy-y^{\prime}x^{\prime}\mid xy=y^{\prime}x^{\prime}\in \Re_0 (r) \}.
\end{equation}
\end{definition}
Clearly $\Acal$ is \emph{a quadratic algebra}, generated by $X$ and
 with defining relations  $\Re(r).$
Furthermore, $\Acal$  is isomorphic to the monoid algebra
$\textbf{k}S(X, r)$.
In many cases the associated algebra
will be \emph{standard finitely presented} with respect to the
 degree-lexicographic ordering induced by an appropriate enumeration of $X$, that is a PBW algebra.
 It is known in particular,  that the algebra $\Acal(\textbf{k},X,r)$
  has remarkable algebraic and homological properties when $r$ is involutive, nondegenerate
 and obeys the braid or Yang-Baxter equation
 in $X\times X\times X$.
 Set-theoretic solutions were introduced in  \cite{D,W} and have been under intensive study during the last decade.
 There are numerous  works on set-theoretic solutions and related
structures, of which a relevant selection for the interested reader
is \cite{W, TM, RTF,  ESS, Lu,  T04, T04s,  Takeuchi, Veselov,
 TSh08, TSh0806, TP}.

\begin{definition}
\label{quantumbinomialset&ybedef} Let $(X,r)$ be a quadratic set.
\begin{enumerate}
 \item
$(X,r)$ is said to be \emph{square-free} if $r(x,x)=(x,x)$ for all $x\in X.$
\item  $(X,r)$ is called a \emph{quantum binomial
set} if it is nondegenerate, involutive and square-free.
\item \label{YBE} $(X,r)$ is a \emph{set-theoretic solution of the
Yang-Baxter equation}
 (YBE)   if  the braid relation
\[r^{12}r^{23}r^{12} = r^{23}r^{12}r^{23}\]
holds in $X\times X\times X.$  In this case $(X,r)$ is also called a \emph{braided set}. If in addition $r$  is involutive $(X,r)$ is called
\emph{a symmetric set}.
\end{enumerate}
\end{definition}
\begin{example}
\label{example1}
$X = \{ \emph{x}_1, x_2, x_3, x_4, x_5 \}$ and $r$ is defined via the actions $\Lcal, \Rcal= (\Lcal)^{-1} \in \Sym(X)$ as
\begin{equation}
\label{examplen5}
\begin{array}{l}
r(x,y) = (\Lcal_x(y), (\Lcal_y)^{-1}(x))\quad\text{where}\\
 \Lcal_{x_1}= \Lcal_{x_3}= (x_2x_4)\quad \Lcal_{x_2}= \Lcal_{x_4}= (x_1x_3)\\
\Lcal_{x_5}= (x_1 x_2 x_3 x_4).
\end{array}
\end{equation}
This is a (square-free) symmetric set. (It has multipermutation
level 3, see \cite{TP}). Presenting the solution $(X,r)$ via the
left and the right actions is an elegant and convenient way to
express the corresponding  $\binom{n}{2}$ quadratic relations
$\Re(r)$ of the algebra $\Acal(\textbf{k}, X, r)$, especially when
$n$ is large. We recommend the reader to write down explicitly the
ten quadratic relations encoded in (\ref{examplen5}). Enumerated
this way, $x_1, \cdots, x_5$ are not PBW generators. If we reorder
the generators as
\[y_1=x_1 \prec y_2=x_3\prec y_3=x_2  \prec  y_4 =  x_4 \prec y_5=x_5,\]
then $y_1, \cdots, y_5$ are PBW generators of  $\Acal$, and
$\Acal$ is a binomial skew-polynomial ring w.r.t this new
enumeration.
\end{example}

\begin{definition}
\label{quantumbinomialalg_def}
A quadratic algebra $A(\textbf{k}, X, \Re) =
\textbf{k}\langle X \rangle/(\Re ) $ is \emph{a quantum
binomial algebra} if
the relations $\Re$ satisfy the following conditions
\begin{enumerate}
\item[(a)]
\label{quantumquadraticdefa} Each relation in  $\Re$ is of the
shape
\begin{equation}
\label{quantumbinomialeq1}
 xy-c_{yx}y^{\prime}x^{\prime},\quad\text{where}\quad x, y, x^{\prime},
y^{\prime} \in X, \quad\text{and}\quad c_{xy} \in \textbf{k}^{\times}
\end{equation}
 (this is what we
call \emph{a binomial relation}).
\item[(b)]
\label{quantumquadraticdefb} Each $xy, x \neq y$ of length $2$
occurs at most once in $\Re$.
\item[(c)]
\label{quantumquadraticdefc} Each relation is \emph{square-free},
i.e. it does not contain a monomial of the shape $xx,$ $x \in X.$
\item[(d)]
\label{quantumquadraticdefd} The relations $\Re$ are \emph{non
degenerate}, i.e. the canonical bijection $r=r(\Re): X\times X
\longrightarrow X \times X,$ associated with $\Re$,  is  non degenerate.
\end{enumerate}
Relations satisfying conditions (a)-(d) are called \emph{quantum binomial relations}.
\end{definition}

Clearly, each binomial skew-polynomial ring is a PBW quantum
binomial algebra. The algebra $\Acal(\textbf{k}, X,r)$ from the
Eaxmple \ref{examplen5} is a concrete quantum binomial algebra.
See  more examples at the end of the section. We recall that
(although this is not part of the definition) every $n$-generated
quantum binomial algebra has exactly $\binom{n}{2}$ relations.

With each quantum binomial algebra we associate two
maps determined canonically via its relations.

\begin{definition}
\label{associatedmapsdef}
Let  $\Re\subset \textbf{k}\langle X
\rangle$ be a set of quadratic binomial relations, satisfying
conditions (a)  and (b). Let $V = Span_{\textbf{k}} X$.

The canonically \emph{associated
quadratic set} $(X,r)$, with
 $r= r(\Re): X\times X \longrightarrow X\times X$
is defined as
\[
r(x,y)=(y^{\prime},x^{\prime}),\; \text{and}\;
r(y^{\prime},x^{\prime})= (x,y), \] if
$xy-c_{xy}y^{\prime}x^{\prime} \in \Re.$ If $xy$  does not occur
in any relation ($x=y$ is  possible)  then we set
\[r(x, y)= (x, y).\]


  \emph{The
automorphism associated with} $\Re,$ $R=R(\Re): V^{\otimes 2}
\longrightarrow V^{\otimes 2},$ is defined analogously. If
$xy-c_{xy}y^{\prime}x^{\prime} \in \Re$, we set
\[
R(x\otimes y)=c_{xy}y^{\prime}\otimes x^{\prime},\;\text{and}\;
R(y^{\prime}\otimes x^{\prime})= (c_{xy})^{-1}x\otimes y.\] If
$xy$ does not occur in any relation
we
set
\[ R(x\otimes y)= x\otimes y.\]

$R$ is called \emph{non-degenerate} if $r$ is non-degenerate. In
this case we shall also say that the defining relations $\Re$ are
\emph{non degenerate binomial relations}.
\end{definition}


Let $V$ be a $\textbf{k}$-vector space. A linear
automorphism $R$ of $V\otimes V$ is  \emph{a solution of the
Yang-Baxter equation}, (YBE)  if the equality
\begin{equation}
\label{YBER} R^{12}R^{23}R^{12} = R^{23}R^{12}R^{23}
\end{equation}
holds in the automorphism  group of $V\otimes V\otimes V,$ where
$R^{ij}$ means $R$ acting on the i-th and j-th component.

A quantum binomial algebra $A = \textbf{k }\langle X ; \Re\rangle,$
is \emph{a Yang-Baxter algebra }, in the sense of Manin
\cite{Maninpreprint}, if the associated map $R = R(\Re)$ is a
solution of the Yang-Baxter equation.

It was shown in \cite{TM} that each binomial skew polynomial ring is a Yang-Baxter algebra.

The results below can be extracted from \cite{TM}, \cite{T96},  and
\cite{T04s}, Theorem B.
\begin{fact}
\label{fact1} Let $A= \textbf{k}\langle X \rangle/(\Re)$ be
a quantum binomial algebra. Then the following two conditions
are equivalent.
\begin{enumerate}
\item
\label{theoremB2} $A$ is a binomial skew polynomial ring, with
respect to some appropriate enumeration of $X.$
\item
\label{theoremB3} The automorphism $R=R(\Re):V^{\otimes 2} \longrightarrow
V^{\otimes 2}$
is a solution of the
 Yang-Baxter equation, so $A$ is a Yang-Baxter algebra.
\end{enumerate}
Each of these conditions implies that $A$ is an Artin-Schelter
regular PBW algebra. Furthermore $A$ is a left and right
Noetherian domain.
\end{fact}
We shall prove in Section \ref{section_AS_YBE} that conversely, in
the class of quantum binomial algebras each Artin-Schelter regular
PBW algebra defines canonically a solution of the YBE, and
therefore is a Yang-Baxter algebra and a binomial skew-polynomial
ring.

We end up the section with two concrete examples of quantum binomial algebras with 4 generators.

\begin{example}
\label{example2}
Let $A  = \textbf{k} \langle x, y, z, t\rangle    /(\Re)$, where $X=\{ x, y, z, t\}$, and
\[
\begin{array}{llllllll}
\Re &=& \{xy-zt, & ty-zx, & xz-yx, &
tz-yt, & xt-tx, & yz-zy\}.
\end{array}
\]
Clearly, the relations are square-free, a direct verification shows that they are nondegenerate.
so $A$ is a quantum binomial algebra.
 More sophisticated proof shows that the set of relations $\Re$ is not a Gr\"{o}bner basis
w.r.t. deg-lex ordering coming from
any order (enumeration) of the set $X$.
This example is studied with details in  Section \ref{section_combinatorics}.
\end{example}

\begin{example}
\label{example3}
Let $A  = \textbf{k} \langle X \rangle/(\Re)$, where $X=\{ x, y, z, t\}$, and
\[
\begin{array}{llllllll}
\Re &=& \{xy-zt, & ty-zx, & xz-yt, &
tz-yx, & xt-tx, & yz-zy\}.
\end{array}
\]
We fix $t>x>z>y$, and take the corresponding deg-lex ordering on $\langle X \rangle$. Then  direct verification shows that $\Re$ is  a Gr\"{o}bner basis.
To do this one has to show that the ambiguities
$txz, \; txy,\; tzy,\; xzy$
are solvable. In this case the set
\[\Ncal = \{y^{\alpha}z^{\beta}x^{\gamma}t^{\delta}\mid \:\alpha,\beta,\gamma,\delta\geq 0\}\]
is the normal basis of $A$, (mod $\Re$).

Note that any order in which $\{t, x,\} > \{z,y\}$, or $\{z,y\} >
\{t, x,\} $, makes $A$ a PBW algebra,  there are exactly eight
such enumerations of $X$.

Furthermore, $A$ is a Yang-Baxter algebra, and an AS-regular domain of global dimension $4$.
\end{example}


\section{PBW algebras with polynomial growth and finite global dimension}
\label{section_growthgldim}

 Let $X = \{x_1, \cdots x_n\}$. As usual, we fix the deg-lex ordering $<$
 on $\langle X\rangle$.
Each element $g\in \textbf{k}\langle X\rangle$ has the shape $g = cu
+ h$, where $u \in \langle X\rangle,$ $c \in \textbf{k}^{\times}, h
\in \langle X\rangle,$ and   either $h = 0$, or $h= \sum_{\alpha}
c_{\alpha} u_{\alpha}$ is a linear combination of monomials
$u_{\alpha} < u$. $u$ is called the leading monomial of $g$ (w.r.t.
$<$) and denoted $LM(g)$. Every finitely presented graded algebra,
$A = \textbf{k}\langle X  \rangle / I $, where $I$ is a finitely
generated
 ideal of $\textbf{k}\langle X  \rangle$
 has an uniquely determined reduced Gr\"{o}bner basis $\textbf{G}$. In general,
 $\textbf{G}$ is infinite.  Anick introduced
  \emph{the set of obstructions $\textbf{W}$ for a connected graded  algebra},
  see \cite{Anick86}. It is easy to deduce from his definition that the set of
  obstructions $\textbf{W}$ is exactly \emph{the set of leading monomials}
of the elements of the reduced
  Gr\"{o}bner basis $\textbf{G}$, that is   $\textbf{W} = \{LM(g)\mid g \in \textbf{G}\}$.
 The obstructions are used to construct a free resolution of
 the field $\textbf{k}$ considered as an $A$-module, \cite{Anick86}.

 Consider now a PBW algebra $A  = \textbf{k}\langle X  \rangle /(\Re)$
with PBW generators \\
$X=\{x_1, \cdots x_n\}$. It follows from the definition of a PBW
algebra that the relations $\Re$ is exactly the reduced
Gr\"{o}bner basis of the ideal $(\Re)$. Hence, in this case,
\emph{the set of obstructions} $\textbf{W}$ is simply the set of
leading monomials of the defining relations.
\[
\textbf{W }= \{ LM(f) \mid f \in \Re\}.
\]

Then  $\textbf{N} = X^2 \backslash \textbf{W}$ is the set of \emph{normal monomials} (mod $\textbf{W}$) of length $2$.

\begin{notation}
\label{notation1}
We set
\[
\begin{array}{c}
\textbf{N}^{(0)}= \{1\},  \quad N^{(1)}= X,\\
\\
\textbf{N}^{(m)}= \{x_{i_1}x_{i_2}\cdots x_{i_m} \mid  x_{i_k}x_{i_{k+1}} \in
\textbf{N},\;\; 1\leq k \leq m-1\}, \; \;  m= 2,3,\cdots \\
\\
\textbf{N}^{\infty}= \bigcup _{m \geq 0}N^{(m)}.
\end{array}
\]
\end{notation}
Note
that the set $\textbf{N}^{(m)}$, $m \geq 0$ is a $\textbf{k}$-basis of
$A_m$, so  $\textbf{N}^{\infty}$ is the set of all normal (mod
$\textbf{W}$) words in $\langle X \rangle$. It is well-known that
the set $\textbf{N}^{\infty}$ project to a basis of $A$. More
precisely, the free associative algebra $\textbf{k}\langle X
\rangle$ splits into a direct sum of
  subspaces
  \[
\textbf{k}\langle X \rangle \simeq  Span_{\textbf{k}}
\textbf{N}^{\infty}\bigoplus I.
\]
So there are isomorphisms of vector spaces
\[\begin{array}{c}
A \simeq Span_{\textbf{k}} \textbf{N}^{\infty}\\
\\
A_m \simeq Span_{\textbf{k}} \textbf{N}^{(m)}, \; \dim A_m =
|\textbf{N}^{(m)}|, \; m= 0, 1, 2, 3, \cdots
\end{array}
\]
  For a PBW algebra $A$ there is a  canonically associated monomial algebra
  $A^0 = \textbf{k}\langle X  \rangle /(\textbf{W})$.
As a monomial algebra, $A^0$ is also PBW. Both algebras $A$ and
$A^0$ have the same set of obstructions $\textbf{W}$ and therefore
they have the same \emph{normal basis} $\textbf{N}^{\infty}$, the
same Hilbert series and the same growth. It follows from results
of Anick that $gl.\dim A = gl.\dim A^0$. More generally, the sets
of obstructions $\textbf{W}$ determines uniquely the Hilbert
series, growth and the global dimension for the whole family of
PBW algebras $A$ sharing the same $W$. The binomial skew
polynomial rings are an well known example of PBW algebras with
polynomial growth and finite global dimension, moreover they are
AS regular Noetherian domains, see \cite{TM}.


Let $M \subset X^2$ be a set of monomials of length $2$. We define
\emph{the graph $\Gamma_M$ corresponding to $M$} as follows.

\begin{definition} $\Gamma_M$   is
\emph{a directed graph} with a set of vertices $V(\Gamma_M) = X$ and
a set of directed edges (arrows) $E= E(\Gamma_M)$ defined as follows
 \[x \longrightarrow y \in E\; \; \text{iff}\; \;
x, y \in X, \; \text{and}\; xy \in M.\]

We recall that \emph{the order } of a graph $\Gamma$ is the number
of its vertices, i.e. $|V(\Gamma )|$, so $\Gamma_M$ is a graph of
order $|X|$. A path of length $k-1$ in $\Gamma_M$  is a sequence of
edges $v_1 \longrightarrow v_2\longrightarrow\cdots v_k $, where
$v_i \longrightarrow v_{i+1}\in E$.

\emph{A cycle (of length $k$)} in $\Gamma$ is a path of the shape
$v_1 \longrightarrow v_2\longrightarrow\cdots v_k \longrightarrow
v_1$ where $v_1, \cdots ,v_k$ are distinct vertices. \emph{A loop }
is a cycle of length $0$,  $x \longrightarrow x.$ So the graph
$\Gamma_M$ contains a loop $x\longrightarrow x$ whenever $xx\in M$,
and a cycle of length two
 $x \longrightarrow y \longrightarrow x$, whenever $xy, yx \in M$. In this case
 $x\longrightarrow y, y\longleftarrow x$ are called bidirected edges.
 Note that, following the terminology in graph theory, \emph{we make difference between directed and oriented graphs}.
A directed graph having no symmetric pair of directed edges (i.e.,
no  pairs $x \longrightarrow y$ and $y \longrightarrow x$) is known
\emph{as an oriented graph}. An oriented graph with no cycles is
called \emph{an acyclic oriented } graph. In particular, such a
graph has no loops.

Denote by $\overline{M}$ the complement $X^2 \backslash M.$ Then the
graph $\Gamma_{\overline{M}}\;$ is \emph{dual to  $\Gamma_M$} in the
sense that
\[x \longrightarrow y \in E(\Gamma_{\overline{M}})\; \:\text{iff}\;
x \longrightarrow y \;\text{is not an edge of }\; \Gamma_M.\]
\end{definition}

Let $A$ be a PBW algebra, let $\textbf{W}$ and $\textbf{N}$  be the set of obstructions, and the set of normal monomials of length $2$, respectively. Then the graph  $\Gamma_\textbf{N}$ gives complete information about the growth of $A$, while   the global dimension of $A$, can be read of $\Gamma_\textbf{W}$.

\emph{The graph of normal words of $A$}, $\Gamma_\textbf{N}$  was introduced in more general context by Ufnarovski \cite{Ufnarovski}.

Note that, in general, $\Gamma_{\textbf{N}}$ is  a directed graph
which may contain pairs of edges,  $x\longrightarrow y,
y\longrightarrow x$, so  $\Gamma_{\textbf{N}}$  is \emph{not
necessarily an oriented graph}.

The following is a particular case of a more general result of Ufnarovski.
\begin{fact} \cite{Ufnarovski}
 \label{GNfact}
 For every $m \geq 1$ there is a one-to-one correspondence between the
 set $\textbf{N}^{(m)}$ of normal words of length
$m$
and the set of paths of length $m-1$ in the graph $\Gamma_{\textbf{N}}$. The path $a_1\longrightarrow a_2\longrightarrow a_2 \longrightarrow \cdots \longrightarrow a_m$ (these are not necessarily distinct vertices) corresponds to the word $a_1a_2\cdots a_m \in \textbf{N}^{(m)}$.
 The algebra $A$ has polynomial growth of degree $m$ \emph{iff}
 \begin{enumerate}
 \item[(i)] The graph $\Gamma_\textbf{N}$ has no intersecting cycles, and
  \item[(ii)] $m$ is the largest number of (oriented) cycles occurring in a path of $\Gamma_\textbf{N}$.
 \end{enumerate}
 \end{fact}

\begin{example}
All binomial skew-polynomial algebras $A$ with five PBW generators
$x_1, x_2, \cdots, x_5$ have the same graph $\Gamma_{\textbf{N}}$
as in Figure 1.
 The graph of obstruction
$\Gamma_{\textbf{W}}$ for $A$ can be seen in Figure 2. The Koszul
dual $A^{!}$ has corresponding graph of normal words
$\Gamma_{\textbf{N}^!}$ represented in Figure 3. The graphs in
Figure 2 and Figure 3 are acyclic tournaments, see Definition
\ref{def_completeacyclicgraph}.
\end{example}

The graph $\Gamma_\textbf{W}$ is dual to  $\Gamma_\textbf{N}$, i.e.
$x \longrightarrow y \in E(\textbf{W})$ \emph{iff} $x\longrightarrow
y$ is not an edge in  $\Gamma_{\textbf{N}}$. Similarly to
$\Gamma_{\textbf{N}}$,  $\Gamma_{\textbf{W}}$ is  a directed graph
and, in general, may contain pairs of edges,  $x\longrightarrow y,
y\longrightarrow x$ or loops $x\longrightarrow x$.

It is straightforward from Anick's general definition of an
$m$-chain, see \cite{Anick85, Anick86}  that in the case of PBW
algebras, each $m$-chain is a monomial of length $m+1$,
$y_{m+1}y_m\cdots y_1,$ where $y_{i+1}y_{i}\in \textbf{W}, \; 1 \leq
i \leq m.$ (For completeness, the \emph{$0$-chains}  are the
elements of $X$, by definition). This implies that for every $m \geq
1$ there is a one-to-one correspondence between the set of
$m$-chains, in the sense of Anick, and the set of paths of length
$m$ in the directed graph $\Gamma_{\textbf{W}}$. The $m$-chain
$y_{m+1}y_m\cdots y_1,$ with $y_{i+1}y_{i}\in \textbf{W}, \; 1 \leq
i \leq m,$ corresponds to the path $y_{m+1}\longrightarrow
y_m\longrightarrow \cdots \longrightarrow y_1$ of length $m$ in
$\Gamma_{\textbf{W}}$.

Note that Anick's resolution \cite{Anick86} \cite{Anick85} is
minimal for PBW algebras, and for monomial algebras (not
necessarily quadratic), and therefore a PBW algebra $A$ has finite
global dimension $d < \infty$ \emph{iff } there is a  $d-1$-chain,
but there are no $d$-chains in $\langle X \rangle$. The following
lemma is a "translation" of this in terms of the properties of
$\Gamma_\textbf{W}$.
 \begin{lemma}
 \label{Gldimlemma}
 $\emph{gl}.\dim A = d < \infty\;$ if and only if  $\; \Gamma_\textbf{W}$ is an acyclic oriented graph, and $d-1$ is the maximal length of a path occurring in $\Gamma_\textbf{W}$.
 \end{lemma}

  All PBW algebras with the same set of PBW generators $x_1, \cdots, x_n$ and the same
  sets of obstructions $\textbf{W}$,  share the same graphs $\Gamma_\textbf{N}$ and
  $\Gamma_\textbf{W}$.
 In some cases it is convenient to study the corresponding monomial
 algebra $A^0$  instead of  $A$.
\begin{figure}[htp]
     \centering
     \includegraphics[scale=0.8]{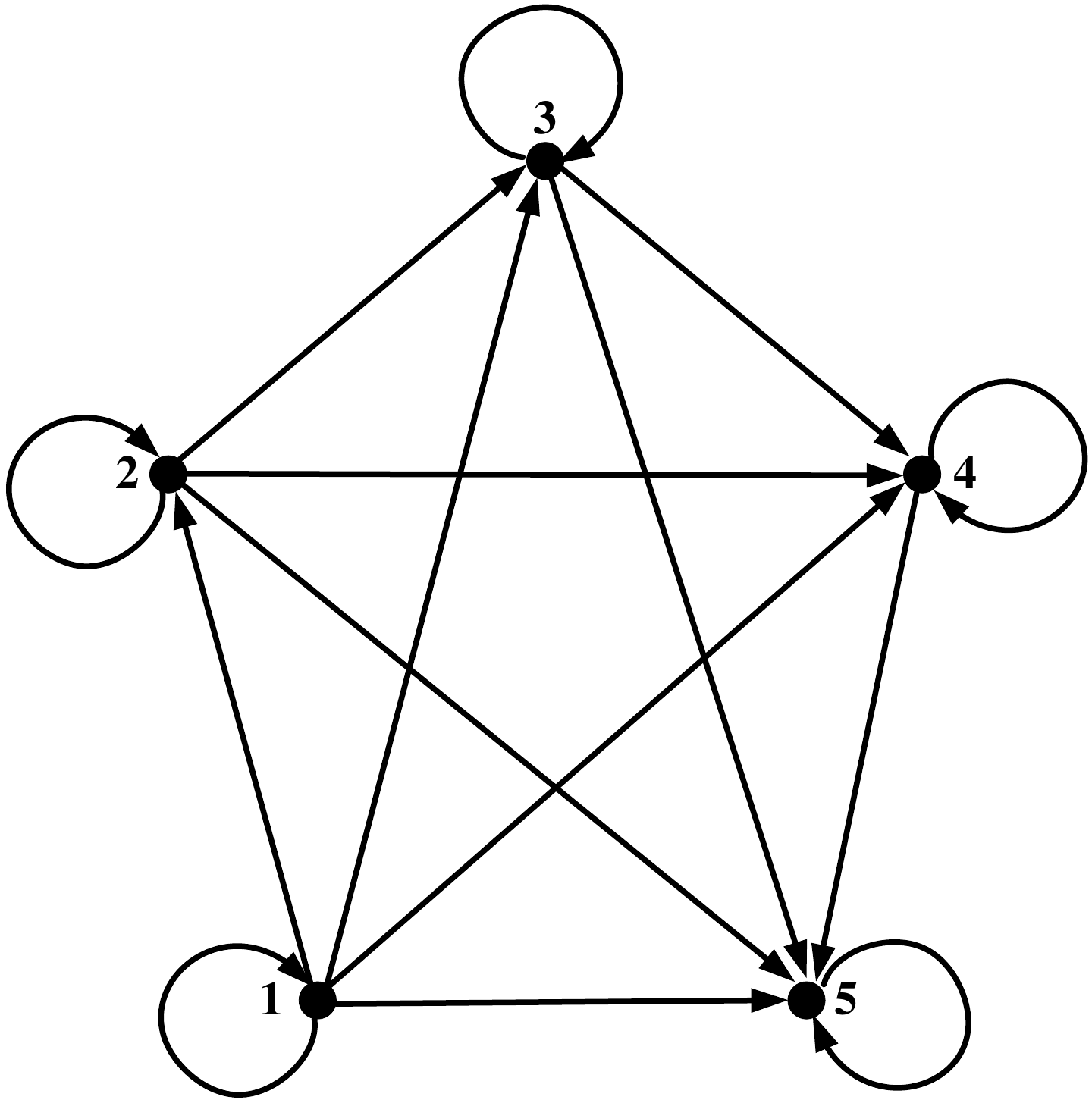}
     \caption{
     This is the graph of normal words $\Gamma_{\textbf{N}}$ for a PBW algebra $A$ with $5$ generators,
     polynomial growth and finite global dimension.}
     \label{fig:1}
\end{figure}


\begin{figure}[htp]
     \centering
     \includegraphics[scale=0.8]{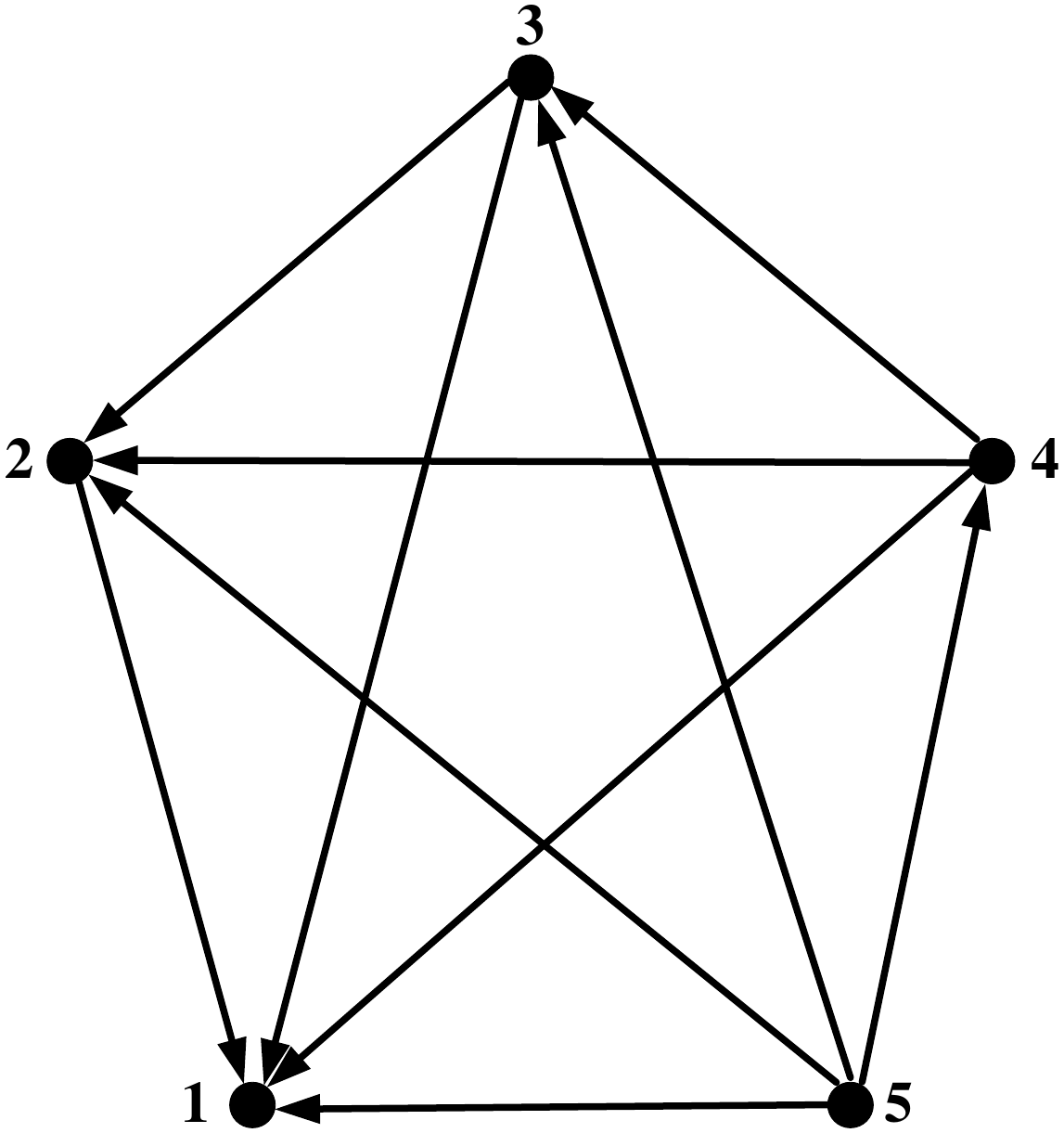}
     \caption{This is the graph of obstructions $\Gamma_{\textbf{W}}$, dual to $\Gamma_{\textbf{N}}$.
     It is \textbf{an acyclic tournament of order 5}, labelled "properly", as in Proposition \ref{beautifulgraphprop}.}
     \label{fig:3}
\end{figure}

\begin{figure}[htp]
     \centering
     \includegraphics[scale=0.8]{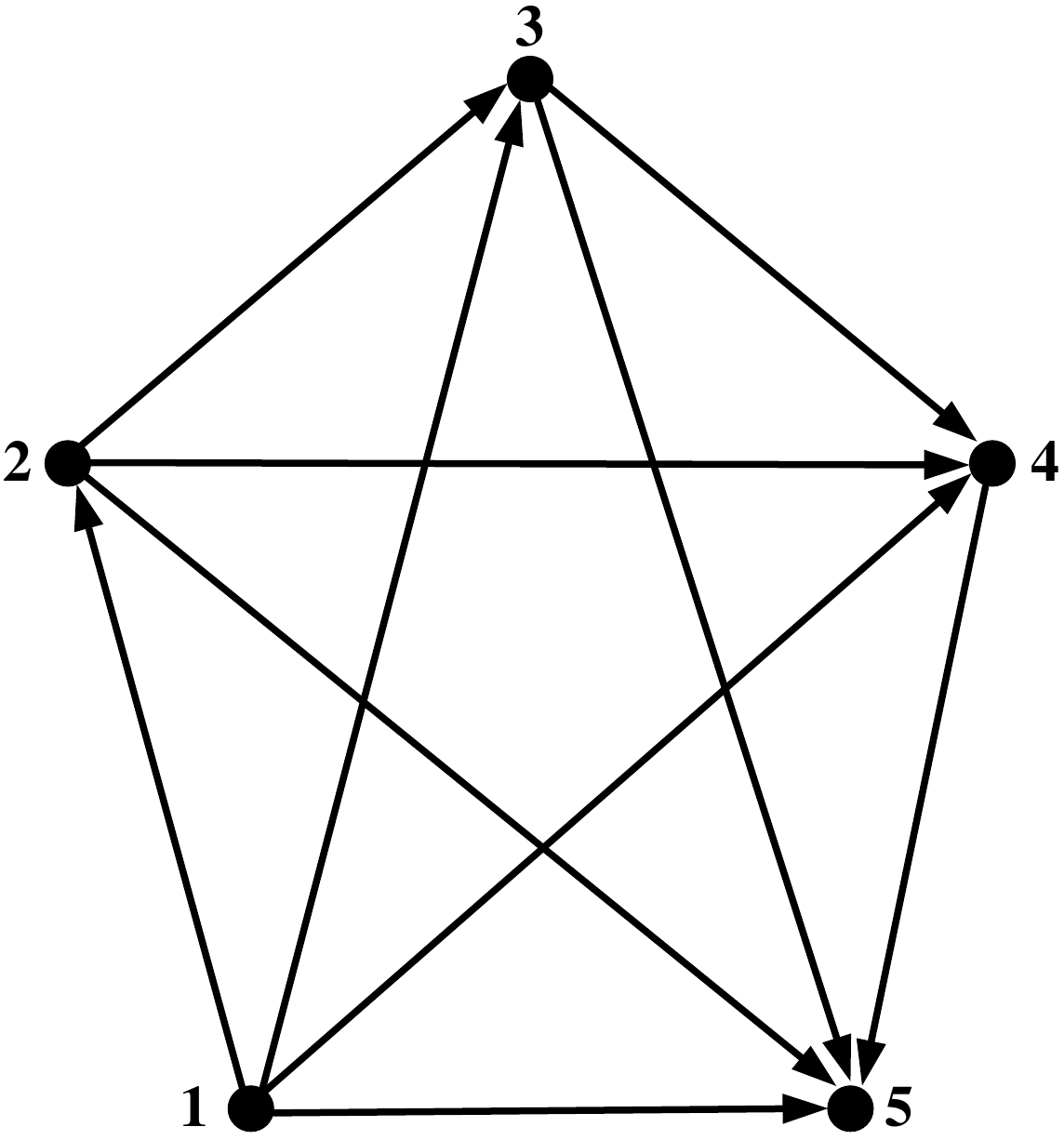}
     \caption{ This is the graph of normal words $\Gamma^{!}$ for the Koszul dual $A^{!}$. It is \textbf{an acyclic tournament of order 5} }.
     \label{fig:2}
\end{figure}
\begin{definition}
\label{def_completeacyclicgraph}
A complete oriented graph $\Gamma$
 is called \emph{a tournament or tour}.
In other words, a tournament is  a directed graph in which each pair
of vertices is joined by a single edge having a unique direction.
 Clearly, a complete directed graph with no cycles (of any length)
 is \emph{an acyclic tournament}.
\end{definition}
The following is straightforward.
\begin{remark}
\label{remar_tournament}
An acyclic oriented graph with $n$ vertices is a tournament \emph{iff } it has exactly $\binom{n}{2}$ (directed) edges.
\end{remark}

We shall need the following  proposition about oriented graphs.
\begin{proposition}
\label{beautifulgraphprop} Let $\Gamma$ be an acyclic tournament
of order $n$. Then the set of its vertices $V=V(\Gamma)$  can be
labelled  $V= \{y_1, y_2, \cdots, y_n\}$, so that the set of edges
is
\begin{equation}
\label{goodlabels}
E(\Gamma)= \{y_j\longrightarrow y_i \mid 1 \leq i < j\leq n\}.
\end{equation}
Analogously, the vertices can be labelled $V= \{z_1, z_2, \cdots,
z_n\}$, so that \[E(\Gamma)= \{z_i\longrightarrow z_j \mid 1 \leq
i < j\leq n\}.\]
\end{proposition}
\begin{proof}
We proof this by induction on the order of $\Gamma$.

The statement is obvious for $n=2.$ Assume the statement of
proposition is true for graphs with $n-1$ vertices. Let $\Gamma$
be an an acyclic tournament, of order $n$ with vertices labelled
$\{1,\cdots, n\}$ and set of edges $E = E(\Gamma)$. Let
$\Gamma_{n-1}$ be the subgraph of $\Gamma$ with set of vertices
$V^{\prime}=
\{1,\cdots, n-1\}$ and set of edges
\[E^{\prime}=
\{i
\longrightarrow j \mid 1 \leq i, j \leq n-1\;
i\longrightarrow j \in E(\Gamma)\}.\]
 By the inductive assumption the
set of vertices $V^{\prime}$ can be relabelled \[V^{\prime}=
\{v_1, \cdots v_{n-1}\}\], s.t.
\[
E^{\prime}= \{v_j\longrightarrow v_i \mid 1 \leq i
< j\leq n-1\}.
\]
Denote by $v$ the $n$-th vertex of $\Gamma$.
Two cases are possible.

(a) $v_j \longrightarrow v \in E(\Gamma),\;  \forall\; 1 \leq j
\leq n-1.$ In this case the relabelling is clear, we set $y_1 =
v,$ and $y_{j+1}= v_j, 1 \leq j \leq n-1$. Then the labelling $V =
\{y_1 \cdots y_n\}$ agrees with \ref{goodlabels}.

(b) there exist a $j,\; 1 \leq j \leq n-1,$ such that $\; v
\longrightarrow v_j \in E(\Gamma).\;$ Let $j$, $1 \leq j \leq n-1$,
be the maximal index   with the property $v \longrightarrow v_j \in
E(\Gamma)$.

Assume $j > 1,$ and let $1 \leq i < j$. We claim that $v
\longrightarrow v_i \in E(\Gamma).$ Assume the contrary. By
assumption the vertices $v, v_i$ are connected with a directed edge,
so $v_i \longrightarrow v \in E(\Gamma).$ Note that by the inductive
assumption $v_j\longrightarrow v_i \in E(\Gamma).$
  So the graph $\Gamma$ contains the cycle
  \[
  v \longrightarrow v_j \longrightarrow v_i \longrightarrow v,
  \]
 but by  hypothesis $\Gamma$ is acyclic, a contradiction.
Thus we have
\[\begin{array}{ll}
v \longrightarrow v_i \in E(\Gamma), &\quad \forall \; i,\; 1 \leq i \leq j,\\
&\\
v_k \longrightarrow v \in E(\Gamma),&\quad  \forall \; k,\; j < k
\leq n-1\quad(\text{if}\; j< n-1).
\end{array}
\]

Three cases are possible

\textbf{A.} $j = 1$. In this case we set \[y_1 = v_1,\;y_2 = v, \;
y_{k+1} = v_k, \; 2 \leq k \leq n-1.\]

\textbf{B.} $1 < j < n-1$. In this case we set \[y_k = v_k, \; 1
\leq k \leq j-1,\; y_j=v, \;y_{k+1} = v_k, \;j \leq k \leq n-1.\]

\textbf{C.} $j = n-1$. In this case we set \[y_k = v_k, \;1 \leq k
\leq n-1, \; y_{n}=v.\]
\end{proof}

$A^0$ is \emph{a quadratic monomial algebra } if it has a
presentation $A^0  = \textbf{k}\langle X  \rangle /(W)$, where $W$
is a set of monomials of length $2$.    Any quadratic monomial
algebra $A^0$ is a PBW algebra, furthermore any enumeration $x_1
\cdots, x_n  $ of $X$ gives PBW generators of $A^0$.

\begin{theorem}
\label{monomialalgebrath} Let $A^0  = \textbf{k}\langle x_1 \cdots,
x_n  \rangle /(\textbf{W})$  be a quadratic monomial algebra. The
following conditions are equivalent
\begin{enumerate}
\item
\label{theor12}
$A^0$ has finite global dimension, and polynomial growth.
\item
\label{theor12a}
$A^0$ has finite global dimension, and $|\textbf{W}|= \binom{n}{2}$.
\item
\label{theor12b}
$A^0$ has polynomial growth, $\textbf{W}\bigcap \diag X^2 = \emptyset$,  and $|\textbf{W}|= \binom{n}{2}$.
 \item
 \label{theor12main}
 The graph $\Gamma_{\textbf{W}}$ is an acyclic tournament.
\item
\label{theor11}
\[
H_{A^0}(z)= \frac{1}{(1-z)^n}.
\]
 \item
 \label{theor13}
 There is a permutation $y_1, \cdots, y_n$ of $x_1 \cdots, x_n$ such that
 \begin{equation}
 \label{Ninf}
 \textbf{N}^{\infty}= \{y_1^{{\alpha}_1}\cdots y_n^{{\alpha}_n}\mid {\alpha}_i \geq 0, \; 1 \leq i \leq n\}.
  \end{equation}
 \item
 \label{theor14}
 There is a permutation $y_1, \cdots, y_n$ of $x_1 \cdots, x_n,$ such that
\[ \textbf{W}= \{y_jy_i \mid \; 1 \leq i < j \leq n\}.
 \]
 \end{enumerate}
 Furthermore, in this case
  \[\emph{gl}.\dim A^0 = n = \;\text{the degree of polynomial growth of } \; A.\]
\end{theorem}
\begin{proof}
Condition (\ref{theor12main}) is central for our proof.

\textbf{A.} We will start with several easy implications.

Suppose (\ref{theor12main}) holds, so $\Gamma_{\textbf{W}}$ is an
acyclic tournament. By Proposition \ref{beautifulgraphprop} the
set of its vertices $V=V(\Gamma_{\textbf{W}})$  can be relabelled
$V= \{y_1, y_2, \cdots, y_n\}$, so that
\begin{equation}
\label{goodlabelsW}
E(\Gamma_{\textbf{W}})= \{y_j\longrightarrow y_i\mid 1 \leq i < j\leq n\}.
\end{equation}
This clearly implies condition (\ref{theor14}). The inverse implication is also clear.

So  (\ref{theor12main}) $\;\Longleftrightarrow\;$ (\ref{theor14}).

The following implications are straightforward
\[\begin{array}{c}
(\ref{theor13}) \;\Longleftrightarrow\; (\ref{theor14}) \; \Longrightarrow \; (\ref{theor11})\\
(\ref{theor14})\Longrightarrow (\ref{theor12b}).
\end{array}
\]
(\ref{theor12main}) $\;\Longrightarrow \;$ (\ref{theor12a}). Suppose (\ref{theor12main}) holds. As an
acyclic tournament  $\Gamma_{\textbf{W}}$ contains exactly
$\binom{n}{2}$ edges, and therefore $|\textbf{W}|= \binom{n}{2}$.
By (\ref{goodlabelsW}) the set $E(\Gamma_{\textbf{W}})$ contains
the edges $y_j\longrightarrow y_{j-1}, 2 \leq j \leq n$, so the
graph $\Gamma_{\textbf{W}}$ has a path $y_n \longrightarrow
y_{n-1}\longrightarrow\cdots \longrightarrow y_1$ of length $n-1$.
Clearly, there are no longer paths in  $\Gamma_{\textbf{W}}$ thus
$gl.\dim A^0 = n$. This verifies (\ref{theor12main})
$\;\Longrightarrow \;$ (\ref{theor12a}).

It is also clear that (\ref{theor12main}) $\;\Longrightarrow \;$
(\ref{theor12}).

(\ref{theor12a}) $\Longrightarrow$ (\ref{theor12main}). Note first
that $|\textbf{W}|= \binom{n}{2}$ implies that the graph has exactly
$\binom{n}{2}$ edges. Next $ gl.\dim A^0 < \infty$ implies that
$\Gamma_{\textbf{W}}$ is an acyclic oriented graph, (see Lemma
\ref{Gldimlemma}) so, by Remark \ref{remar_tournament},
$\Gamma_{\textbf{W}}$ is an acyclic tournament, we have shown
\[
(\ref{theor12a}) \Longleftrightarrow
(\ref{theor12main})\Longrightarrow (\ref{theor12}).
\]

(\ref{theor12b}) $\Longrightarrow$ (\ref{theor12main}). Assume
(\ref{theor12b}) holds. Then $\Gamma_{\textbf{W}}$ has exactly
$\binom{n}{2}$ edges, furthermore  each edge has the shape
$x\longrightarrow y,$  $x \neq y$. Therefore its dual graph
$\Gamma_{\textbf{N}}$ has a loop $x \longrightarrow x$ at every
vertex, and exactly $\binom{n}{2}$ edges $y\longrightarrow x,$ where
$x\longrightarrow y\in E(\Gamma_{\textbf{W}})$. The polynomial
growth of $A^0$ implies that $\Gamma_{\textbf{N}}$ has no cycles of
length $\geq 2$, and therefore every two vertices in
$\Gamma_{\textbf{N}}$ are connected with a single directed edge, so
$\Gamma_{\textbf{N}}$ is an oriented graph. It follows then that
$\Gamma_{\textbf{W}}$ is an acyclic oriented tournament, which
verifies the implication (\ref{theor12b}) $\Longrightarrow$
(\ref{theor12main}). The inverse implication is clear. We have shown
\[
(\ref{theor12main})\Longleftrightarrow (\ref{theor12b}).
\]

It remains to show (\ref{theor12})  $\;\Longrightarrow \;$
(\ref{theor12main}), and (\ref{theor11})  $\;\Longrightarrow \;$
(\ref{theor12main}). The two implications are verified by similar
argument.

\textbf{B.}  (\ref{theor11})   $\;\Longrightarrow\;$ (\ref{theor12main}).
Note first that
\begin{equation}
\label{hilbeq}
H_{A^0}(z)= \frac{1}{(1-z)^n} = 1 + nz + \binom{n+1}{2}z^2 + \binom{n+2}{3}z^3 + \cdots.
\end{equation}
So
\[
\dim A_2 =|\textbf{N}| = \binom{n+1}{2}, \quad \text{which implies}\quad |\textbf{W}| = \binom{n}{2}.
\]
Secondly,  the special shape of Hilbert series $H_{A^0}(z)$ implies
that $A^0$ has polynomial growth of degree $n$. Therefore by Fact
\ref{GNfact} the graph $\Gamma_{\textbf{N}}$ contains a path with
$n$ cycles. The only possibility for such a path is to have $n$
distinct  vertices and a loop at every vertex:

\bigskip

\bigskip

\begin{equation}
\label{longestpatheq}
\xymatrixrowsep{.8pc}
\xymatrix @!=0pc {
&*{\bullet}="b" \ar@(ul,ur) c &*{\longrightarrow}
&*{\bullet}="b" \ar@(ul,ur) c &*{\longrightarrow}
&*{\bullet}="b" \ar@(ul,ur) c &*{\longrightarrow}
&*{\cdots}="b"  &*{\longrightarrow}
&*{\bullet}="b" \ar@(ul,ur) c \\
&*{^{a_1}}&&*{^{a_2}}&&*{^{a_3}}&&&&*{^{a_n}} (10,30)
}
\end{equation}

\bigskip
Indeed $\Gamma_{\textbf{N}}$ has exactly $n$ vertices, and has no
intersecting cycles. Each loop $x\longrightarrow x$ in
$\Gamma_{\textbf{N}}$ implies  $xx \in N$ , so $\Delta_2 \subset
\textbf{N}$ ($\Delta_2 = \diag(X^2)$). Then the complement
$\textbf{N}\backslash \Delta_2$ contains exactly
 $\binom{n}{2}$  monomials of the shape $xy, x \neq y$, or equivalently
 $\Gamma_{\textbf{N}}$ has $\binom{n}{2}$ edges of the shape $x \longrightarrow y, x \neq y.$
 Clearly,
 no pair  $x \longrightarrow y, y \longrightarrow x$ belongs to $E(\Gamma_{\textbf{N}})$,
 otherwise $\Gamma_{\textbf{N}}$ will have two intersecting cycles $x\longrightarrow x$ and
 $x \longrightarrow y \longrightarrow x$, but this is impossible
 since $A^0$ has polynomial growth.
Therefore $\Gamma_{\textbf{N}}$ is an oriented graph.

 Consider now the dual graph $\Gamma_{\textbf{W}}$. The properties of $\Gamma_{\textbf{N}}$
 imply that (a) $\Gamma_{\textbf{W}}$ has no loops. (b) $\Gamma_{\textbf{W}}$ has no cycles of length
 $\geq 2$. Each edge $x\longrightarrow y$ in $E(\Gamma_{\textbf{N}})$ has a corresponding edge $x\longleftarrow y \in E(\Gamma_{\textbf{W}})$.
 So $\Gamma_{\textbf{W}}$ is an acyclic oriented graph with $\binom{n}{2}$ edges, and
 Remark \ref{remar_tournament} again implies that it  is an acyclic tournament.
 This proves (\ref{theor11})   $\;\Longrightarrow\;$ (\ref{theor12main}).

 \textbf{C. } Finally we show
(\ref{theor12})  $\;\Longrightarrow \;$ (\ref{theor12main}).

Assume that $A^0$ has  polynomial growth and finite global
dimension. We shall  use once more the nice balance between the dual
graphs $\Gamma_{\textbf{W}}$ and $\Gamma_{\textbf{N}}$. Note first
that $\Gamma_{\textbf{N}}$ has no intersecting cycles,  since
otherwise $A$ would have exponential growth. On the other hand
$\Gamma_{\textbf{W}}$ is acyclic, therefore it is an acyclic
oriented graph. In particular, $\Gamma_{\textbf{W}}$  has no loops,
or equivalently $\textbf{W}$ does not contain monomials of the type
$xx$. It follows then that the dual graph $\Gamma_{\textbf{N}}$ has
loops $x \longrightarrow x$ at every vertex. Secondly, for each pair
$x \neq y$ of vertices, there is exactly one edge $x \longrightarrow
y,$ or $y \longrightarrow x,$ in $\Gamma_{\textbf{N}}$. Indeed, $x
\longrightarrow y, y \longrightarrow x \in E(\Gamma_{\textbf{N}})$
would imply that $\Gamma_{\textbf{N}}$ has intersecting cycles,
which is impossible. Moreover if there is no edge connecting $x$ and
$y$ in $\Gamma_{\textbf{N}}$ this would imply that both $x
\longrightarrow y$, $y \longrightarrow x$ are edges  of
$\Gamma_{\textbf{W}}$, hence $\Gamma_{\textbf{W}}$ has a cycle
$x\longrightarrow y\longrightarrow x,$  which contradicts
$gl.\dim A < \infty.$

We have shown that  $\Gamma_{\textbf{W}}$ is an acyclic oriented
graph with $\binom{n}{2}$ edges, and
 Remark \ref{remar_tournament} again implies that it  is an acyclic tournament.
\end{proof}
\begin{remark}
The implication (\ref{theor12}) $\;\Longrightarrow\;$ (\ref{theor11}) follows straightforwardly from a result of Anick, see \cite{Anick85} Theorem 6.
\end{remark}

\begin{prooftheorem1}
Assume now that $A  = \textbf{k}\langle X  \rangle /(\Re)$ is a
quadratic PBW algebra, with PBW generators $X=\{x_1, \cdots
x_n\}$. Let $\textbf{W }$ be the set of obstructions, and let
$A^0=\textbf{k}\langle X  \rangle /(\textbf{W})$ be the
corresponding monomial algebra. The set $\textbf{N}^{\infty}$,
(Notation \ref{notation1}) is a  $\textbf{k}$-basis  for both
algebras $A$ and $A^0$. As we have noticed before, the two
algebras have the same Hilbert series, equal degrees of growth,
and by Lemma \ref{Gldimlemma}, there is an equality $gl. \dim A =
gl. \dim A^0.$

(\ref{t1gldimpolgr})  $\Longrightarrow$ (\ref{t1perm}).
Suppose $A$ has finite global dimension and polynomial growth. Then the same is valid for
$A^0$. By Theorem \ref{monomialalgebrath}.\ref{theor13}
 there is a permutation $y_1, \cdots, y_n$ of $x_1 \cdots, x_n,$ such that
 \[\textbf{N}^{\infty}= \{y_1^{{\alpha}_1}\cdots y_n^{{\alpha}_n}\mid {\alpha}_i \geq 0, \; 1 \leq i \leq n\},\]
so $A$ has a $\textbf{k}$-basis of the desired form. (In general, it is not true that $\textbf{N}^{\infty}$
is \emph{a normal basis} for $A$ w.r.t. the deg-lex ordering $\prec$  defined via $y_1\prec\cdots \prec y_n$).

(\ref{t1perm}) $\Longrightarrow$ (\ref{t1hilbser}) is clear.

(\ref{t1hilbser}) $\Longrightarrow$ (\ref{t1gldimpolgr}) and
(\ref{t1hilbser}) $\Longrightarrow$ (\ref{t1gldim}).

Assume (\ref{t1hilbser}) holds.
Then obviously $A$ has polynomial growth of degree $n$.
The equalities
\[
H_{A^0}(z)= H_A(z)= \frac{1}{(1-z)^n}
\]
and  Theorem \ref{monomialalgebrath} imply that the monomial algebra $A^0$
has $gl.\dim A^0= n$,  and  $\mid W\mid = \binom{n}{2}.$  Clearly, then $A$ has $\binom{n}{2}$ relations and global dimension $n$.
This gives the implications (\ref{t1hilbser}) $\Longrightarrow$ (\ref{t1gldimpolgr}),
and (\ref{t1hilbser}) $\Longrightarrow$ (\ref{t1gldim})

Similarly, condition (\ref{t1gldim}) is satisfied simultaneously by
$A$ and $A^0$, so, by Theorem \ref{monomialalgebrath},
(\ref{t1hilbser}) holds. We have shown
(\ref{t1gldim})$\Longleftrightarrow$ (\ref{t1hilbser}).

The theorem has been proved.
\end{prooftheorem1}

\section{Combinatorics in quantum binomial sets}
\label{section_combinatorics}
In this section $(X,r)$ is a finite quantum binomial set.

When we study the monoid $S=S(X,r)$,   or the algebra $\Acal =
\Acal(\textbf{k}, X,r) \simeq \textbf{k }[ S]$ associated with
$(X,r)$ (see Definition \ref{associatedobjects}), it is convenient
to use the action of the infinite groups, $\Dcal_m(r)$, generated
by maps associated with the quadratic relations, as follows. We
consider the  bijective maps
\[
r^{ii+1}: X^m \longrightarrow X^m, \; 1 \leq i \leq m-1,
\quad\text{where}\quad r^{ii+1}=Id_{X^{i-1}}\times r\times
Id_{X^{m-i-1}}.
\]
Note that these maps are elements of the symmetric group $\Sym
(X^m)$. Then the group $\Dcal _m(r)$ generated by  $r^{ii+1}, \; 1
\leq i \leq m-1,$ acts on $X^m.$ $r$ is involutive, so the
bijective maps $r^{ii+1}$ are involutive, as well, and
$\Dcal_m(r)$ is the infinite group
\begin{equation}
\label{Dk}
 \Dcal_m(r)=\: _{\rm{gr}} \langle r^{ii+1}\mid\quad (r^{ii+1})^2= e,\quad  1 \leq i \leq m-1 \rangle. \end{equation}

When $m=3$, we  use notation $\Dcal=\Dcal_3(r)$. Note that
\[\Dcal=\: _{\rm{gr}} \langle r^{ii+1}\mid\quad (r^{ii+1})^2=
e,\quad  1 \leq i \leq 2 \rangle\]
is simply \emph{the infinite
dihedral group}.

\emph{The problem of equality of words} in the monoid $S=S(X,r)$ and
in the quadratic algebra $\Acal$ is solvable. Two elements $\omega,
\omega^{\prime}\in \langle X\rangle$ are equal in $S$ \emph{iff}
they have the same length, $|\omega|= |\omega^{\prime}|=m$ and
belong to the same orbit of $\Dcal_m(r)$ in $X^m.$

The action of the infinite dihedral group  $\Dcal$ on $X^3$ is of particular importance in this
section. Assuming that $(X,r)$ is a quantum binomial set we will
find some counting formulae and inequalities involving the orders of
the $\Dcal$-orbits in $X^3$, and their number, see Proposition
\ref{proposition1}. These are used to find a necessary and
sufficient conditions   for $(X,r)$ to be a symmetric set,
 Proposition \ref{proposition2}, and to give upper bounds for
$\dim A_3$ and $\dim A^{!}_3$ in the general case of quantum binomial algebra $A$,  Corollary \ref{bounddimA3cor}.

As usual, the orbit of a monomial $\omega \in X^3$ under the  action of $\Dcal$
will be denoted by  $\Ocal= \Ocal(\omega)$.

Denote by $\Delta_i$ the diagonal of $X^{\times i}, 2\leq i \leq 3$.
One has $\Delta_3 = (\Delta_2 \times X) \bigcap (X \times \Delta_2)$.
\begin{definition}
\label{squarefreeorbitdef}
We  call a $\Dcal$-orbit $\Ocal$ \emph{square-free} if
\[\Ocal\bigcap (\Delta_2 \times X \bigcup X \times \Delta_2) = \emptyset.\]
A monomial $\omega\in X^3$ is \emph{square-free} in $S$ if its
orbit $\Ocal(\omega)$ is square-free.
\end{definition}
\begin{remark}
\label{remark_actions}
We recall that whenever $(X,r)$ is a quadratic set, the left and the right "actions"
\[{}^z{\bullet} : X\times X \longrightarrow X\quad\text{and}\quad \bullet^z : X\times X \longrightarrow X\] induced by $r$
reflect each property of $r$, see \cite{TSh08}, and \cite{T10},
Remark 2.1. We will need the following properties of the actions.
Assume that $r$ is square-free and nondegenerate, then
\begin{equation}
\label{rightactioneq1}
\begin{array}{lllll}
{}^zt={}^zu &\Longrightarrow &t = u &\Longleftarrow & t^z= u^z
\\
{}^zt = z&\Longleftrightarrow &t = z &\Longleftrightarrow &t^z =
z.
\end{array}
\end{equation}
\end{remark}

\begin{lemma}
\label{squarefreeorbitslemma}
Let $(X,r)$ be a quantum binomial set, and let $\Ocal$ be a square-free $\Dcal$-orbit in $X^3$.
Then $|\Ocal| \geq 6.$
\end{lemma}
\begin{proof}
Suppose  $\Ocal= \Ocal(xyz)$ is a square-free orbit. Consider the set \[ O_1=  \{ v_i \mid 1 \leq i \leq 6 \}\subseteq \Ocal\] consisting of the first six elements of the  "Yang-Baxter" diagram
\begin{equation}
\label{ybediagram}
\begin{CD}
v_1= xyz \quad\quad\quad @>r^{12}>> \quad\quad\quad ({{}^xy}x^y)z = v_2\\
@V  r^{23} VV @VV r^{23} V\\
  v_3=x({{}^yz}y^z)@. ({{}^xy}) ({}^{x^y}z)(x^y)^z= v_5\\
@V r^{12} VV @VV r^{12} V\\
\kern -80pt v_4={}^{x}{({}^yz)}(x^{{}^yz})(y^z) @.
\kern-100pt \quad\quad \quad\quad\quad\quad\quad\quad\quad\quad    [{}^{{}^xy}{({}^{x^y}z)}][({}^xy)^{({}^{x^y}z)}][(x^y)^z]= v_6.
 \end{CD}
\end{equation}
Clearly,
\[
O_1 = U_1\bigcup U_3\bigcup U_5, \quad\text{where}\quad U_j = \{v_j, \; r^{12}(v_j) = v_{j+1}\},\quad j = 1,3,5.
\]
We claim that $U_1, U_3, U_5$ are pairwise disjoint sets, and each of them has order 2.
Note first that since  $v_j$ is  a square-free monomial, for each $j = 1, 3, 5$, one has
$v_j\neq r_{12}(v_j)= v_{j+1}$,
therefore
\[
|U_j| = 2, \;\; j = 1,3,5.
\]
The monomials in each $U_j$ have the same "tail". More precisely,
$v_1 = (xy)z, v_2=r(xy)z$, have a "tail" $z$, the tail of $v_3$,
and $v_4$ is $y^z$, and the  tail of $v_5$, and $v_6$ is
$(x^y)^z$. It will be enough to show that the three elements $z,
y^z,  (x^y)^z \in X$ are pairwise distinct. But $\Ocal(xyz)$ is
square-free, so $y\neq z$ and by (\ref{rightactioneq1}) $y^z \neq
z$.
 Furthermore $v_2 = ({}^xy)(x^y)z \in \Ocal(xyz)$ and therefore, $x^y \neq y$ and $x^y \neq z$. Now by (\ref{rightactioneq1}) one has
\[\begin{array}{lll}
x^y \neq z &\Longrightarrow &(x^y)^z \neq z\\
x^y \neq y &\Longrightarrow & (x^y)^z \neq y^z.
\end{array}
\]
We have shown that the three elements $z, y^z, (x^y)^z \in X$
occurring as tails in $U_1, U_3, U_5$, respectively,  are pairwise
distinct, so the three sets are pairwise disjoint. This implies
$|O_1|= 6$, and therefore $|\Ocal| \geq 6.$
\end{proof}
\begin{proposition}
\label{proposition1} Suppose $(X,r)$ is a finite quantum
binomial set.  Let  $\Ocal$  be a $\Dcal$-orbit in $X^3$, denote $\textbf{E(\Ocal)} = \Ocal
 \bigcap((\Delta_2 \times X\bigcup X \times \Delta_2)\backslash \Delta_3)$ .
\begin{enumerate}
\item
\label{prop11}
The following implications hold.
\label{orb1}
\[
\begin{array}{lllll}
&\emph{\textbf{(i)}}&\quad \Ocal \bigcap \Delta_3 \neq \emptyset&\Longrightarrow& |\Ocal|=1.\\
&&&&\\
&\emph{\textbf{(ii)}}&\quad \textbf{E(\Ocal)}
\neq \emptyset & \Longrightarrow& |\Ocal| \geq 3 \quad \text{and}\quad |\textbf{E(\Ocal)}|=2.
  \end{array}
  \]
\\
In this case we say that \emph{\emph{$\Ocal$ is an orbit of type \textbf{\textbf{(ii)}}}}\\
\[
\begin{array}{lllll}
&\emph{\textbf{(iii)}}&\quad \Ocal
 \bigcap(\Delta_2 \times X\bigcup X \times \Delta_2) = \emptyset&
 \Longrightarrow & |\Ocal| \geq 6.
  \end{array}
\]
\item \label{prop12} There are exactly $n(n-1)$ orbits $\Ocal$ of
type \emph{\textbf{(ii)}} in $X^3$. \item \label{prop1cc} $(X,r)$
satisfies the cyclic condition \emph{iff} each orbit $\Ocal$ of
type \emph{\textbf{(ii)}} has order $|\Ocal| = 3.$ In this case
\[
{}^{x^y}y = {}^xy\quad\text{and}\quad  x^{{}^xy}=x^y , \; \forall
x,y \in X.
\]
\end{enumerate}
\end{proposition}
\begin{proof}
Clearly, the "fixed" points under the action of $\Dcal$ on $X^3$
are exactly the monomials $xxx,$ $x \in X,$ this gives
\textbf{(i)}.

Assume now that $\Ocal$ is of type  \textbf{(ii)}. Then it
contains an element of the shape $\omega = xxy,$ or $\omega =
xyy$, where $x, y \in X, x \neq y.$ Without loss of generality we
can assume $\omega = xxy \in \Ocal.$

The orbit $\Ocal (\omega)$  can be obtained
as follows.
We fix as an initial element of the orbit $\omega = xxy.$  Then there is \emph{a unique  finite sequence}
   $r^{23}, r^{12}, r^{23}, \cdots,$  which exhaust the whole orbit, and produces  a "new" element at every step.
    $r$ is involutive and square-free, thus in order to produce new elements at every step, the sequence must start with $r^{23}$ and at every next step we have to alternate the actions
   $r^{23}$, and $r^{12}$.

   We look at the "Yang-Baxter" diagram starting with $\omega$ and exhausting the whole orbit (without repetitions).
   \begin{equation}
\label{orb2}
\omega=\omega_1 = xxy \longleftrightarrow^{r^{23}} \omega_2 = x ({}^xy) (x^y) \longleftrightarrow^{r^{12}}  \omega_3= ({}^{x^2}y) (x^{{}^xy}) (x^y) \longleftrightarrow \cdots \longleftrightarrow \omega_m.
\end{equation}
Note first that the first three elements $\omega_1, \omega_2,
\omega_3$ are distinct monomials in $X^3$. Indeed, $x \neq y$
implies $r(xy)\neq xy$ in $X^2$, so $\omega_2 \neq \omega_1$. By
assumption $(X,r)$ is square-free, so ${}^xx=x$, but by the
nondegeneracy $y \neq x,$ also implies  ${}^xy \neq x.$  So
$r(x({}^xy))\neq x({}^xy),$ and therefore $\omega_3 \neq
\omega_2.$ Furthermore, $\omega_3 \neq \omega_1.$ Indeed, if we
assume $x = {}^{x^2}y = {}^x{({}^xy)}$ then by
(\ref{rightactioneq1}) one has ${}^xy = x$, and therefore $y = x,$
a contradiction. We have obtained that $\mid\Ocal\mid \geq 3.$

We claim now that the intersection $\textbf{E}= \textbf{E(\Ocal)}$
contains exactly two element.
We analyze the diagram (\ref{orb2}) looking from left to right.

Suppose we have made $k-1$ "steps" to the right obtaining new elements, so we have obtained
\[
\omega_1 = xxy \longleftrightarrow^{r^{23}} \omega_2 = x {}^xy x^y
\longleftrightarrow^{r^{12}}  \omega_3= {}^{x^2}y x^{{}^xy} x^y
\longleftrightarrow \cdots\longleftrightarrow
\omega_{k-1}\longleftrightarrow^{r^{ii+1}} \omega_k,
\]
where $\omega_1, \omega_2, \cdots, \omega_k$ are pairwise distinct.
Note that all elements $\omega_s$, $\; 2 \leq s \leq k-1,$ have the
shape $\omega_s= a_sb_sc_s,$ with $a_s\neq b_s$ and $b_s\neq c_s$.
Two cases are possible.

\textbf{(a)} $\omega_k = a_kb_kc_k,$ with $a_k\neq b_k$ and $b_k\neq
c_k$, then applying $r^{jj+1}$ (where $j=2$ if  $i = 1$ and  $j = 1$
if $i=2$), we obtain a new element $\omega_{k+1}$ of the orbit.

\textbf{(b) }
 $\omega_k = aac,$ or  $\omega_k = acc,$ $a\neq c$.
 In this case  $r^{jj+1} $ with $j \neq i$ keeps $\omega_k$ fixed,
 so the process of obtaining new elements of the orbit stops at this step and the diagram is complete.

 But our diagram is finite, so  as a final step on the right it has to "reach" some $\omega_m = aac,$ or $\omega_m = acc,$ $a\neq c$ (we have already shown that $m \geq 3$). Note that  $\omega_m \neq \omega_1$.
Hence the intersection $\textbf{E}= \textbf{E}(\Ocal)$, contains
exactly two elements.

Condition
\textbf{(iii)}  follows from Lemma \ref{squarefreeorbitslemma}.
(\ref{prop11}) has been verified.

We claim that there exists exactly $n(n-1)$ orbits of type \textbf{(ii)}.
Indeed, let $\Ocal_1, \cdots \Ocal_p$ be all orbits of type \textbf{(ii)}.
The intersections
$E_i= E(\Ocal _i), 1 \leq i \leq p,$
are disjoint sets and each of them contains two elements. Now the equalities
\[
\begin{array}{lll}
\bigcup_{1 \leq i \leq p} E_i &= &(\Delta_2 \times X\bigcup X \times \Delta_2)\backslash \Delta_3\\
&&\\
\mid E_i\mid = 2, &\quad& \mid(\Delta_2 \times X\bigcup X \times \Delta_2)\backslash \Delta_3\mid= 2n(n-1)\\
\end{array}
\]
imply $p = n(n-1)$. This verifies (\ref{prop12})

Condition (\ref{prop1cc}) follows straightforwardly  from (\ref{orb2}).
 \end{proof}
\begin{example}
\label{example2_orbits} Consider the quantum binomial algebra $A$
given in example \ref{example2}. Let $(X,r)$ be the associated
quadratic set, $S = S(X,r)$ the corresponding monoid. The relations
are semigroup relations, so there is an algebra isomorphism $A =
\Acal(\textbf{k}, X, r) \simeq \textbf{k} S.$ We will find the
corresponding $\Dcal$ orbits in $X^3$. There are 12 orbits of type
\textbf{(ii)}. This agrees with Proposition \ref{proposition2}.
 \[\begin{array}{l}
 \begin{CD}
 \Ocal_1=xxy @>r^{23}>> xzt@>r^{12}>>yxt@>r^{23}>> ytx @>r^{12}>>tzx @>r^{23}>>tty
 \end{CD}
 \\
 \\
\begin{CD}
  \Ocal_2=  xxz  @>r^{23}>>xyx @>r^{12}>> ztx@>r^{23}>> zxt@>r^{12}>>tyt@>r^{23}>> ttz
  \end{CD}
 \\
 \\
 \begin{CD}
 \Ocal_3= yxx@>r^{12}>>xzx@>r^{23}>>xty @>r^{12}>>txy@>r^{23}>> tzt @>r^{12}>>ytt
 \end{CD}
 \\
 \\
\begin{CD}
 \Ocal_4=  zxx@>r^{12}>>tyx@>r^{23}>>txz @>r^{12}>>xtz@>r^{23}>> xyt @>r^{12}>> ztt
  \end{CD}
 \end{array}
  \]
  \[
  \begin{array}{ll}
   \\
   &
   \\
 \begin{CD}
  \Ocal_5=xyy@>r^{12}>>zty
 @>r^{23}>>zxz;
  \end{CD}&
 \begin{CD}
 \Ocal_6=xzz @>r^{12}>>yxz
 @>r^{23}>>yyx\end{CD}
 \\
 &
 \\
\begin{CD} \Ocal_7=
 tyy@>r^{12}>>zxy@>r^{23}>>zzt;
 \end{CD}&
\begin{CD}
  \Ocal_8=tzz@>r^{12}>>ytz@>r^{23}>>yyt
 \end{CD} \\& \\
\begin{CD}
  \Ocal_9=txx @>r^{12}>>xtx@>r^{23}>>xxt; \end{CD} &
\begin{CD}
  \Ocal_{10}=xtt@>r^{12}>>txt@>r^{23}>>ttx
 \end{CD}
 \\&
 \\
\begin{CD}
\Ocal_{11}=  yzz@>r^{12}>>zyz@>r^{23}>>zzy;
 \end{CD} &
 \Ocal_{12}=\begin{CD}
 zyy @>r^{12}>>yzy@>r^{23}>>yyz
 \end{CD}
 \end{array}
\]
There are only two square-free orbits, $\Ocal^{(1)}=\Ocal(xyz)$, and $\Ocal^{(2)}=\Ocal(tyz)$. Each of them has order $6$.
 \begin{equation}
\label{ybediagram2}
\begin{CD}
 xyz  @>r^{12}>>  ztz\\
@V  r^{23} VV @VV r^{23} V\\
  xzy@. zyt\\
@V r^{12} VV @VV r^{12} V\\
yxy  @>r^{12}>>  yzt
 \end{CD}\quad\quad\quad\quad\quad\quad\quad\quad\quad \begin{CD}
 tyz  @>r^{12}>>  zxz\\
@V  r^{23} VV @VV r^{23} V\\
  tzy@. zyx\\
@V r^{12} VV @VV r^{12} V\\
yty  @>r^{12}>>  yzx
 \end{CD}
\end{equation}
The one element orbits are $\{xxx\}, \{yyy\}, \{zzz\}, \{ttt\}$.

More detailed study of the orbits shows that $A$ is \emph{not PBW w.r.t. any enumeration of $X$}.
Clearly, $r$ does not satisfy the braid relation, so  $(X,r)$ is not a symmetric set.
\end{example}

\begin{lemma}
\label{YBE=correctorderoforbits}
A quantum binomial set
$(X,r)$ is symmetric \emph{iff } the orders of $\Dcal$-orbits $\Ocal$ in $X^3$ satisfy the following two conditions.
\[
\begin{array}{lll}
\emph{(a)}&\quad \Ocal
 \bigcap(\Delta_2 \times X\bigcup X \times \Delta_2)\backslash \Delta_3)) \neq \emptyset & \Longleftrightarrow |\Ocal| =  3.\\
 & &\\
\emph{(b)}&\quad \Ocal
 \bigcap(\Delta_2 \times X\bigcup X \times \Delta_2) = \emptyset&
\Longleftrightarrow |\Ocal| = 6.
  \end{array}
\]
\end{lemma}
\begin{proof}
Look at the corresponding YBE diagrams
\end{proof}
Let $(X,r)$ be a quantum binomial set, let  $\Dcal$ be the infinite dihedral group
acting on $X^3$. We fix the following notation for the $\Dcal$-orbits in $X^3$.
\begin{notation}
\label{notation2}
We denote by  $\Ocal_i, 1 \leq i \leq n(n-1)$ the orbits of type \textbf{(ii)}, and by
$\Ocal^{(j)}, 1 \leq j \leq q$ all
square-free orbits in  $X^3$. The remaining $\Dcal$-orbits in $X^3$ are the one-element orbits
$\{xxx\}, x \in X$, their union is $\Delta_3$.
\end{notation}

\begin{proposition}
\label{proposition2}
Let  $(X,r)$ be a finite quantum
binomial set. Let $\Ocal^{(j)}, 1 \leq j\leq q,$ be the set of all (distinct) square-free $\Dcal$-orbits in $X^3.$ Then
\begin{enumerate}
\item
\label{prop21}
  $q\leq \binom{n}{3}.$
\item
\label{prop22}
$(X,r)$ is a symmetric set \emph{iff} $\;q =  \binom{n}{3}$.
\end{enumerate}
\end{proposition}
\begin{proof}
Clearly, $X^3$ is a disjoint union
\[X^3 = \Delta_3 \bigcup_{1 \leq i \leq n(n-1)}\Ocal_i \bigcup_{1 \leq j \leq q} \Ocal^{(j)} .\]
Thus
\begin{equation}
\label{countingformulae2}
\mid X^3\mid = \mid \Delta_3\mid  + \sum_{1 \leq i \leq n(n-1)} \mid\Ocal_i\mid +\sum_{1 \leq j \leq q} \mid\Ocal^{(j)}\mid
\end{equation}
Denote $m_i = |\Ocal_i|,\; 1 \leq i \leq n(n-1)$, $\; n_j = |\Ocal^{(j)}|, \; 1 \leq j \leq q.$
By Proposition \ref{proposition1} one has
\[m_i \geq 3, \;1 \leq i \leq n(n-1),\quad\text{and}\quad n_j \geq 6, \;1 \leq j \leq q.\]
We replace these inequalities in (\ref{countingformulae2}) and obtain
\begin{equation}
\label{countingformulae3}
n^3 = n + \sum_{1 \leq i \leq n(n-1)} m_i + \sum_{1 \leq j \leq q} n_j \geq  \; n + 3n(n-1) + 6 q.
\end{equation}
So \[q \leq \frac{n^3 -3n^2 +2n}{6}= \binom{n}{3},\] which verifies
(\ref{prop21}). Assume now $q= \binom{ n}{3}$. Then
(\ref{countingformulae3}) implies
\[
n^3 = n + \sum_{1 \leq i \leq n(n-1)} m_i + \sum_{1 \leq j \leq \binom{n}{3}} n_j
 \geq  \; n + 3n(n-1) + 6\binom{n}{3} = n^3.
\]
This is possible \textbf{iff } the following equalities hold
\begin{equation}
\label{countingformulae4}
\begin{array}{l}
m_i= |\Ocal_i| = 3, \;1 \leq i \leq n(n-1),\\
\\
n_j = |\Ocal^{(j)}| = 6, \; 1 \leq j \leq q.\\
 \end{array}
 \end{equation}
By Lemma \ref{YBE=correctorderoforbits}, the equalities  (\ref{countingformulae3}) hold \emph{iff  } $(X,r)$ is  a symmetric set.
\end{proof}
\begin{corollary}
\label{cor1_prop2} Let  $(X,r)$ be a finite quantum binomial set.
$(X,r)$ is a symmetric set \emph{iff} the associated quadratic algebra $\Acal= \Acal(\textbf{k}, X,r)$
satisfies \[\dim
\Acal_3= \binom{n+2}{3}.\]
 \end{corollary}
\begin{proof}
 The distinct elements of the associated monoid $S=S(X,r),$ form a $\textbf{k}$- basis of the
monoidal algebra $\textbf{k} S \simeq \Acal(\textbf{k}, X,r)$. In
particular $\dim \Acal_3$ equals the number of distinct monomials of
length $3$ in $S$ which is exactly the number of $\Dcal$-orbits in
$X^3$.
\end{proof}

There is a  close relation between Yang-Baxter monoids and a
special class of \emph{Garside monoids}, see \cite{Chouraqui}, and
\cite{T10}. Garside monoids and groups were introduced by Garside,
\cite{garside}. The interested reader can find more information
and references in \cite{garside, dehornoy, KasselT}, etall.

In \cite{T10}, Definition 1.10. we introduce \emph{the regular
Garside monoids}. It follows from \cite{T10},the  Main Theorem
1.16,  that a finite quantum binomial set $(X,r)$ is a solution of
YBE \emph{iff} the associated monoid  $S = S(X, r)$ is a regular
Garside monoid. This together with Corollary \ref{cor1_prop2}
imply the following.

\begin{corollary}
\label{corGars_prop2} Let  $(X,r)$ be a finite quantum binomial
set. Let $S=S(X,r)$ be
the associated monoid, and  $S^3$-the set of distinct elements of
length $3$ in $S$. Suppose the cardinality of $S^3$ is
\[|S^3| = \binom{n+2}{3}.\] Then $S$ is a Garside monoid. Moreover,
$S$ is regular in the sense of \cite{T10}.
 \end{corollary}

Assume now that $A$ is a quantum binomial algebra. We want to
estimate the
 dimension $\dim A_3$.
Let $(X,r)$ be the corresponding quantum binomial set, $S = S(X,r)$,
$\Acal= \Acal(\textbf{k}, X,r).$ We use Proposition
\ref{proposition2} to find an upper bound for the number of distinct
$\Dcal$-orbits in $X^3$, or equivalently, the order of  $S_3$, the
set of (distinct) elements of length $3$ in $S$. One has
 \[|S_3|= n + n(n-1) +q \leq
 n+n(n-1)+ \binom{n}{3}= \binom{n+2}{3}.\]
There is an isomorphism of vector spaces, $\Acal_3 \simeq Span S_3
$, so
\[\dim \Acal_3 = |S_3| \leq \binom{n+2}{3}.\] In the general case,
a quantum binomial algebra, satisfies $\dim A_3 \leq \dim \Acal_3,$
due to the coefficients $c_{xy}$ appearing in the set of relations.
We have proven the following corollary.
\begin{corollary}
\label{bounddimA3cor}
If $A$ is a quantum binomial algebra, then
\[\dim A_3 \leq \binom{n+2}{3},\quad \dim A^{!}_3 \leq \binom{n}{3}.\]
\end{corollary}

\section{Quantum binomial algebras. Yang-Baxter equation and Artin-Schelter regularity}
\label{section_AS_YBE}
\begin{definition}
\label{frobeniusalgebra}\cite{Maninpreprint}, \cite{Maninbook} A
graded algebra $A=\bigoplus_{i\ge0}A_i$ is called
\emph{a Frobenius algebra of dimension $n$},
(or \emph{a Frobenius quantum space of dimension $n$})
if
\begin{enumerate}
\item[(a)]
$dim(A_n)=1$,  $A_i =0,$ for $i>n.$
\item[(b)]
For all $n\geq j \geq 0$) the multiplicative map
$m: A_j\otimes A_{n -j} \rightarrow A_n$
is a perfect duality (nondegenerate pairing).

A Frobenius algebra $A$ is called \emph{a quantum Grassmann algebra} if in addition
\item[(c)]
$dim_ \textbf{k}A_i= \binom{n}{i} , \;\text{for}\; 1 \leq i \leq
n.$ \end{enumerate}
\end{definition}

\begin{lemma}
\label{proposition33} Let $A  = \textbf{k}\langle X; \Re\rangle$  be
a quantum binomial algebra, $|X|=n$, $A^{!}$ its Koszul dual. Let
$(X, r), r=r(\Re)$ be the associated quantum binomial set (see
Definition \ref{associatedmapsdef}). Then each of the following
three conditions implies that $(X,r)$ is a symmetric set.
\begin{enumerate}
\item
\label{dimA3correct} \[\dim A_3= \binom{n+2}{3}.\]
\item
\label{nchose3_squarefreemonomials} \[\dim A^{!}_3= \binom{n}{3}.\]
\item
\label{N3correct}
 $X$ can be enumerated $X = \{x_1 \cdots, x_n\}$, so that the set of ordered monomials of length 3
\begin{equation}
\label{eqN3}
N_3 = \{x_{i_1}x_{i_2}x_{i_3}\mid 1 \leq i_1 \leq i_2 \leq i_3\leq n\}
\end{equation}
projects to a
$\textbf{k}$-basis of $A_3$
\end{enumerate}
\end{lemma}
\begin{proof}
In the  usual notation, $S=S(X,r)$ and $\Acal =\Acal(X,r)$ denote
respectively be the associated monoid and quadratic algebra. We know
that, in general $\dim A_3 \leq \dim \Acal_3$.

  Assume
(\ref{prop22}) holds. Consider the relations
\[
\binom{n+2}{3}=\dim A_3 \leq \dim \Acal_3 \leq \binom{n+2}{3},
\]
where the right-hand side inequality follows from Corollary
\ref{cor1_prop2}. This implies $\dim\Acal_3 = \binom{n+2}{3}$, or
equivalently, $|S_3|=\binom{n+2}{3}$ and therefore there are exactly
$q=\binom{n}{3}$ square-free $\Dcal$-orbits in $X^3$. Then
Proposition \ref{proposition2} (\ref{prop22}) implies that $(X,r)$
is a symmetric set, which verifies
 (\ref{dimA3correct})$ \Longrightarrow $ (\ref{N3correct}).
The converse (\ref{N3correct}) $ \Longrightarrow $ (\ref{dimA3correct}) is straightforward.
Finally, one has
\[\dim A_3=
\binom{n+2}{3} \Longleftrightarrow \dim A^{!}_3=
\binom{n}{3},\]
This can be proved directly using the $\Dcal$-orbits in $X^3$. It is also
straightforward from the following formula
for quadratic algebras, see \cite{PP}, p 85.
\[
\dim A^{!}_3=(\dim A_1)^3-2(\dim A_1)(\dim A_2)+ \dim A_3.
\]
\end{proof}

\begin{lemma}
\label{skewtyperel&dimA3lemma}
Let $A $ be a quadratic algebra with relations of skew-polynomal type, let $A^{!}$ be its Koszul dual.
Then the following  conditions are equivallent.
\begin{enumerate}
\item
\label{A3}
\[\dim A_3= \binom{n+2}{3}.\]
\item
\label{A!3}
\[\dim A^{!}_3= \binom{n}{3}.\]
\item
\label{A}
The set of defining relations $\Re$ for $A$ is a Gr\"{o}bner basis, so  $A$ is a skew polynomial ring,
and therefore a PBW algebra.
\item
\label{A!}
The set of defining relations $\Re^{\bot}$ for $A^{!}$ is a Gr\"{o}bner basis, so  $A^{!}$ is a
a PBW algebra with PBW generators  $x_1, \cdots, x_n$.
\end{enumerate}
\end{lemma}
\begin{sketchofproof}
As we already mentioned $(\ref{A3})\Longleftrightarrow (\ref{A!3}).$

$(\ref{A3})\Longleftrightarrow \ref{A}).$ Assume condition
(\ref{A3}) holds. The set of monomials of length $3$, which are
normal mod the ideal $(\Re)$ form a $\textbf{k}$-basis of $A_3$. The
skew-polynomial shape of the relations implies that each  monomial
$u$ which is normal mod the ideal $(\Re)$ is in the set of ordered
monomials
\[N_3 = \{ x_{i_1}x_{i_2}x_{i_3}\mid 1 \leq i_1\leq i_2 \leq i_3\leq n\}.\]
There are equalities
\[
\mid N_3\mid = \binom{n+2}{3} =  \dim A_3,
\]
hence all ordered monomials of length $3$ are normal mod  $(\Re)$. It follows then that all ambiguities $x_kx_jx_i, 1 \leq i < j <k \leq n$ are resolvable, and
 by Bergman's Diamond lemma \cite{Bergman}, $\Re$ is a Gr\"obner basis of the ideal $(\Re).$
 Thus $A$ is a PBW algebra, more precisely, $A$ is a binomial skew-polynomial ring. This gives $(\ref{A3})\Longrightarrow (\ref{A}).$
 The converse implication $(\ref{A3})\Longleftarrow (\ref{A})$ is clear.

$(\ref{A!3})\Longleftarrow (\ref{A!})$ is clear.

 $(\ref{A!3})\Longrightarrow (\ref{A!})$ is analogous to $(\ref{A3})\Longrightarrow \ref{A}).$
\end{sketchofproof}

\begin{theorem}
\label{theorem2}
Let $A  = \textbf{k}\langle X  \rangle /(\Re)$ be a quantum binomial algebra, $|X|= n$, and let
 $\Re$ be the associated automorphism $R=R(\Re): V^{\otimes 2}
\longrightarrow V^{\otimes 2},$ see Definition \ref{associatedmapsdef}.
Then
the following three conditions are equivalent
\begin{enumerate}
\item
\label{theorem21}
$R = R(\Re)$ is a solution of the Yang-Baxter equation.
\item
\label{theorem22}
$A$ is a binomial skew-polynomal ring
\item
\label{theorem23}
\[\dim_{\textbf{k}}A_3 = \binom{n+2}{3}.\]
\end{enumerate}
\end{theorem}
\begin{proof}
We start with the implication
(\ref{theorem23}) $\Longrightarrow$ (\ref{theorem22}).

Assume that $\dim_{\textbf{k}}A_3 = \binom{n+2}{3}$. Consider the
corresponding quadratic set $(X,r)$ and the monoidal algebra $\Acal
= \Acal(\textbf{k}, X,r)$. As before we conclude that
$\dim_{\textbf{k}}\Acal_3 = \binom{n+2}{3}$ and therefore, by
Corollary \ref{cor1_prop2} $(X,r)$ is symmetric set.

By  \cite{T04} Theorem 2.26 there exists an ordering on $X$, $X=
\{x_1,x_2,\cdots,x_n\}$ so that the algebra $\Acal$ is a
skew-polynomial ring and therefore a PBW algebra, with PBW
generators $x_1,x_2,\cdots,x_n$. It follows then that the
relations $\Re$ of $A$ are relations of skew-polynomial type, that
is condition (a), (b) and (c) of Definition \ref{binomialringdef}
are satisfied.

By assumption $\dim_{\textbf{k}}A_3 = \binom{n+2}{3}$ so Lemma \ref{skewtyperel&dimA3lemma} implies that
the set of relations $\Re$ is a Gr\"{o}bner basis ($A$ is PBW)
and therefore $A$ is a binomial skew polynomial ring. This verifies (\ref{theorem23}) $\Longrightarrow$ (\ref{theorem22}).

It follows from Definition
\ref{binomialringdef} that  if $A$ is a binomial skew-polynomial ring then the set of monomials
\[\Ncal = \{x_{i_1}x_{i_2}x_{i_3}\mid i_1 \leq x_{i_2}\leq x_{i_3}\}\]
is a $\textbf{k}$-basis of $A_3$, so $\dim_{\textbf{k}}A_3 =
\binom{n+2}{3}$, hence (\ref{theorem22}) $\Longrightarrow$
(\ref{theorem23}).

The equivalence  (\ref{theorem21}) $\Longleftrightarrow$ (\ref{theorem22}) is proven in  \cite{T04s}, Theorem B.
\end{proof}

\begin{prooftheorem3}

The equivalence of conditions
(\ref{theorem34}),  (\ref{theorem35}) and (\ref{theorem36}) follows from Theorem \ref{theorem2}.

The implication (\ref{theorem37}) $\Longrightarrow$ (\ref{theorem36}) is clear.
The converse follows from (\ref{theorem36}) $\Longrightarrow$ (\ref{theorem35}) $\Longrightarrow$ (\ref{theorem37}).

It is straightforward that each binomial skew polynomial ring $A$ is
Koszul and satisfies  (\ref{theorem31})  and (\ref{theorem32}). It
is proven in \cite{T04s} that the Koszul dual $A^{!}$ of a binomial
skew polynomial ring is a quantum Grassman algebra. Thus
(\ref{theorem35}) $\Longrightarrow$ (\ref{theorem38}).

Clearly, (\ref{theorem38}) $\Longrightarrow$ (\ref{theorem36}), and
therefore (\ref{theorem35}) $\Longleftrightarrow$ (\ref{theorem38}).

It is also known that a Koszul algebra $A$ is Gorenstein \emph{iff} its dual $A^{!}$ is Frobenius.
It follows then that  (\ref{theorem35}) $\Longrightarrow$ (\ref{theorem33}).

The result that every binomial skew polynomial ring is AS regular follows also
from the earlier
work  \cite{TM}.

Finally we will show (\ref{theorem31}) $\Longrightarrow$ (\ref{theorem35}) and (\ref{theorem32}) $\Longrightarrow$ (\ref{theorem35}).

We know that a quantum binomial algebra $A$ has exactly $\binom{n}{2}$ relations.
 Assume now that $A$ is a PBW algebra. Then its set of obstructions $\textbf{W}$
 has order
 $|\textbf{W}|=|\Re| = \binom{n}{2}$. Consider now the corresponding monomial algebra
 $A^0 = \textbf{k}\langle X \rangle/ (\textbf{W})$. Each of the conditions  (\ref{theorem31}) and (\ref{theorem32}) is satisfied by $A$ \emph{iff}  it is true for $A^0$.

 Assume first that $A$ satisfies (\ref{theorem31}). Then the monomial algebra $A^0$ satisfies condition
(\ref{theor12a}) of  Theorem \ref{monomialalgebrath} and by the same theorem
 the Hilbert series of $A^0$ satisfies (\ref{theorem37}).
 The algebras $A$ and $A^0$ have the same Hibert series, and therefore
 \[(\ref{theorem31})\Longrightarrow (\ref{theorem37}) \Longrightarrow
 (\ref{theorem35}).\]

Similarly, if $A$ satisfies (\ref{theorem32}), then the monomial
algebra $A^0$ satisfies condition (\ref{theor12b}) of  Theorem
\ref{monomialalgebrath}. (Note that the relations are square free,
so \[\textbf{W}\bigcap\diag X^2 = \emptyset.\] Analogous to the
previous implication we conclude
(\ref{theorem32})$\Longrightarrow$ (\ref{theorem37})
$\Longrightarrow$ (\ref{theorem35}). The equivalence of the
conditions (\ref{theorem31}) $\cdots$ (\ref{theorem38}) has been
verified.

Each of these conditions imply that $A$ is a binomial skew-polynomial ring and therefore
it is a Noetherian domain, see \cite{TM}, or Fact \ref{fact1}.
\end{prooftheorem3}

{\bf Acknowledgments}.
I express my gratitude to Mike Artin, who inspired my research in this area, for his
encouragement and moral support through the years.
This paper was written during my visit as a Fellow of the
Abdus Salam International Centre for Theoretical Physics (ICTP),
Trieste, Summer  2010. It is my pleasant duty to thank  Ramadas Ramakrishnan and the
Mathematics group of ICTP for inviting me and for our valuable and stimulating discussions. I thank ICTP for
the continuous support and for the inspiring and creative
atmosphere.

\end{document}